\documentclass[a4paper]{amsart}
\usepackage{amsfonts,amsmath,amsthm,amssymb,color}
\usepackage[all]{xy}

\usepackage[colorlinks=true, linkcolor=blue, citecolor=blue, urlcolor=blue]{hyperref}

\numberwithin{equation}{section}

\begin{document}

\newtheorem{theorem}{Theorem}[section]
\newtheorem{lemma}[theorem]{Lemma}
\newtheorem{proposition}[theorem]{Proposition}
\newtheorem{corollary}[theorem]{Corollary}

\theoremstyle{definition}
\newtheorem{definition}[theorem]{Definition}
\newtheorem{example}[theorem]{Example}
\newtheorem{notation}[theorem]{Notation}

\theoremstyle{remark}
\newtheorem{remark}[theorem]{Remark}
\newtheorem*{ack}{Acknowledgments}

\newcommand{\Set}{\mathbf{Set}}
\newcommand{\Art}{\mathbf{Art}}
\newcommand{\solose}{\Rightarrow}

\renewcommand{\bar}{\overline}
\newcommand{\de}{\partial}
\newcommand{\eps}{\varepsilon}
\newcommand{\debar}{{\overline{\partial}}}
\newcommand{\per}{\!\cdot\!}

\newcommand{\Oh}{\mathcal{O}}
\newcommand{\sA}{\mathcal{A}}
\newcommand{\sB}{\mathcal{B}}
\newcommand{\sC}{\mathcal{C}}
\newcommand{\sD}{\mathcal{D}}
\newcommand{\sE}{\mathcal{E}}
\newcommand{\sF}{\mathcal{F}}
\newcommand{\sG}{\mathcal{G}}
\newcommand{\sH}{\mathcal{H}}
\newcommand{\sI}{\mathcal{I}}
\newcommand{\sJ}{\mathcal{J}}
\newcommand{\sK}{\mathcal{K}}
\newcommand{\sL}{\mathcal{L}}
\newcommand{\sM}{\mathcal{M}}
\newcommand{\sP}{\mathcal{P}}
\newcommand{\sU}{\mathcal{U}}
\newcommand{\sV}{\mathcal{V}}
\newcommand{\sX}{\mathcal{X}}
\newcommand{\sY}{\mathcal{Y}}
\newcommand{\sN}{\mathcal{N}}
\newcommand{\sZ}{\mathcal{Z}}

\newcommand{\CE}{\operatorname{CE}}
\newcommand{\Aut}{\operatorname{Aut}}
\newcommand{\Mor}{\operatorname{Mor}}
\newcommand{\Def}{\operatorname{Def}}
\newcommand{\Hom}{\operatorname{Hom}}
\newcommand{\HOM}{\operatorname{\mathcal H}\!\!om}
\newcommand{\DER}{\operatorname{\mathcal D}\!er}
\newcommand{\Spec}{\operatorname{Spec}}
\newcommand{\Der}{\operatorname{Der}}
\newcommand{\End}{{\operatorname{End}}}
\newcommand{\END}{\operatorname{\mathcal E}\!\!nd}
\newcommand{\Image}{\operatorname{Im}}
\newcommand{\coker}{\operatorname{coker}}
\newcommand{\tot}{\operatorname{tot}}
\newcommand{\ten}{\otimes}
\newcommand{\mA}{\mathfrak{m}_{A}}
\newcommand{\dec}{\mbox{d\'ec}}
\newcommand{\op}{\operatorname}

\renewcommand{\Hat}[1]{\widehat{#1}}
\newcommand{\dual}{^{\vee}}
\newcommand{\desude}[2]{\dfrac{\de #1}{\de #2}}
\newcommand{\A}{\mathbb{A}}
\newcommand{\N}{\mathbb{N}}
\newcommand{\R}{\mathbb{R}}
\newcommand{\Z}{\mathbb{Z}}
\renewcommand{\H}{\mathbb{H}}
\renewcommand{\L}{\mathbb{L}}
\newcommand{\proj}{\mathbb{P}}
\newcommand{\K}{\mathbb{K}\,}
\newcommand\C{\mathbb{C}}
\newcommand\T{\mathbb{T}}
\newcommand\Del{\operatorname{Del}}
\newcommand\Tot{\operatorname{Tot}}
\newcommand\Grpd{\mbox{\bf Grpd}}
\newcommand\vr{``}
\newcommand{\rh}{\rightarrow}
\newcommand{\contr}{{\mspace{1mu}\lrcorner\mspace{1.5mu}}}

\newcommand{\bi}{\boldsymbol{i}}
\newcommand{\bl}{\boldsymbol{l}}
\newcommand{\bk}{\boldsymbol{k}}

\newcommand{\tr}{\triangleright}
\newcommand{\tl}{\triangleleft}
\newcommand{\MC}{\operatorname{MC}}
\newcommand{\Coder}{\operatorname{Coder}}
\newcommand{\TW}{\operatorname{TW}}
\newcommand{\id}{\operatorname{id}}
\newcommand{\ad}{\operatorname{ad}}
\newcommand{\cone}{\operatorname{C}}
\newcommand{\cocone}{C}
\newcommand{\cylinder}{\operatorname{Cyl}}
\newcommand{\Yuk}{\operatorname{Yu}}
\newcommand{\Grass}{\operatorname{Grass}}
\newcommand{\rank}{\operatorname{rank}}

\title{Algebraic models of local period maps and Yukawa algebras}
\author{Ruggero Bandiera}
\address{\newline
Universit\`a degli studi di Roma La Sapienza,\hfill\newline
Dipartimento di Matematica \lq\lq Guido
Castelnuovo\rq\rq,\hfill\newline
P.le Aldo Moro 5,
I-00185 Roma, Italy.}
\email{bandiera@mat.uniroma1.it}

\author{Marco Manetti}
\email{manetti@mat.uniroma1.it}
\urladdr{www.mat.uniroma1.it/people/manetti/}

\date{\today}

\begin{abstract}
We describe some $L_{\infty}$ model for the local period map of a compact K\"{a}hler manifold. 
Applications include the study of deformations with associated variation of Hodge structure 
constrained by certain closed strata of the Grassmannian of the de Rham cohomology. 
As a byproduct we obtain an interpretation in the framework of deformation theory of the Yukawa coupling.
\end{abstract}

\maketitle

\section*{Introduction}
For a better understanding of the content of this paper
it is preferable to begin with  a heuristic and somewhat imprecise description.
Let $X$ be a compact K\"{a}hler manifold of dimension $n$;
the Hodge filtration on the de Rham complex of $X$ determines a filtration of graded vector spaces
\[ 0\subset F^nH^*(X,\C)\subset F^{n-1}H^*(X,\C)\subset\cdots \subset
F^{0}H^*(X,\C)=H^*(X,\C)=\oplus_i H^i(X,\C)\]
and therefore a sequence of elements in the total Grassmannian
$\Grass(H^*(X,\C))$, i.e., in the (reducible) variety of all graded vector subspaces of the de Rham cohomology.
We recall that $\Grass(H^*(X,\C))$ is smooth and the Zariski tangent space at a point $A^*$ is
\[T_{A^*}\Grass(H^*(X,\C))=\prod_{i=0}^n\Hom\left(A^i,\frac{H^i(X,\C)}{A^i}\right)\;.\]

Let $\sX\to (U,0)$ be a deformation of $X$, with $U$ sufficiently small and contractible; then topologically $\sX$ is the product $U\times X$, in particular for every $u\in U$ there exists a natural
isomorphism $H^*(X_u,\C)=H^*(X_0,\C)$, induced by the Gauss-Manin connection, and it is well known that the  local
$i$th period map
\[ \sP^i\colon U\to \Grass(H^*(X,\C)),\qquad
\sP^i(u)=
F^{i}H^*(X_u,\C),\]
is holomorphic, cf. \cite[Thm. 10.9]{Voisin}. Since Grassmannians admits natural stratifications, e.g. via  Schubert varieties, it is natural to consider deformations whose period map is constrained into a fixed closed stratum. 
For instance, given
$\sX\to (U,0)$ as above one can consider the  determinantal loci

\begin{equation}\label{equ.definizioneluoghi}
\Sigma_{i,j,p,q}(U)=\left\{u\in U\;\middle|\; \rank\left( F^{i}H^p(X_u,\C)\to \frac{H^p(X_0,\C)}{F^{j}H^p(X_0,\C)}\right)\le q\right\}\,.
\end{equation}

Using Griffiths' description of the differential of the period map, it is straightforward to
calculate the Zariski tangent space of $\Sigma_{i,j,p,q}(U)$ in terms of the Kodaira-Spencer map
$T_{0}U\to H^1(X,\Theta_X)$; transversality implies in particular
that
$T_0\Sigma_{i,j,p,q}(U)=T_0U$ for every $i>j$.  On the other hand, if
$j<i\le p<n+j$  it is natural to expect
$\Sigma_{i,j,p,q}(U)$ to be a closed analytic proper subset  of $U$, provided $q$ sufficiently small and
$\sX\to (U,0)$  sufficiently general.
Therefore the determinantal loci $\Sigma_{i,j,p,q}(U)$ are generally singular even for very general families $\sX\to (U,0)$, 
and a  study of them
requires a deep understanding of higher derivatives of the local period maps.

In this paper, we restrict our attention to the subset $Y_U=\Sigma_{n,1,n,0}(U)$, that can be written more explicitly as 
\[Y_U=
\{u\in U\mid  F^{n}H^n(X_u,\C)\subseteq F^{1}H^n(X_0,\C)\}=\{u\in U\mid  F^{n}H^*(X_u,\C)\subseteq F^{1}H^*(X_0,\C)\},\] 
although we think
that ideas and methods developed here can be applied in a more or less straightforward way to more general situations.
This choice is also motivated to a better understanding of the Yukawa coupling, as defined in the algebraic geometry framework 
(\cite[p. 132]{CGGH},\cite[p. 36]{Green}), from the point of view of deformation theory and infinitesimal variations of Hodge structures. 

Since every proper flat family of complex manifolds is obtained locally as a pull back of the Kuranishi family $\sX\to B$ of the central fibre $X_0$,  it is sufficient to study the locus $Y_B\subset B$. 
In doing this there are several nontrivial  problems to solve. The first is that the classical description 
of the period map is defined only for families over reduced basis, while there are many examples of projective manifolds with non-reduced Kuranishi base space \cite{vakil}. In fact, for families over certain non reduced singularities the Gauss-Manin connection does not integrate (this explains, for instance, the assumption on the map $d$ in  \cite[Prop. 4.2]{bloch}).     

The solution to this problem that we adopt here is based on the Fiorenza-Manetti 
description of the period map given in \cite{FMperiods,FMAJ}. Very briefly (details in Sections~\ref{sec.CHF} and \ref{sec:FM}), if $A_{X_0}^{*,*}$ denotes the de Rham complex of $X_0$ and 
$A_X^{0,p}(\Theta_{X_0})$ the space of $(0,p)$-forms with values in the holomorphic tangent bundle, then 
a deformation of $X_0$ over $B$ is induced by a holomorphic family $\xi\colon B\to A_X^{0,p}(\Theta_{X_0})$ of integrable almost complex structures, defined up to gauge equivalence \cite{GoMil2,ManRendiconti}. 
The contraction operator
\[ \bi_{\xi}\colon  A_{X_0}^{*,*}\to A_{X_0}^{*,*},\qquad \bi_{\xi}(\omega)=\xi\contr\omega,\]
has bidegre $(-1,1)$ and in particular is nilpotent. The integrability condition $d\xi+\frac{1}{2}[\xi,\xi]=0$ implies that, for every $p$, the image 
$e^{\bi_{\xi}}(A_{X_0}^{\ge p,*}\otimes \Oh_B)$ is a subcomplex, their cohomology classes inside $H^*(X_0,\C)\otimes \Oh_B$
are gauge invariant and, when the bases $B$ is reduced, we recover the usual variation of Hodge structure
$H^*(e^{\bi_{\xi}}(A_{X_0}^{\ge p,*}\otimes \Oh_B))=F^pH^*(X_{\xi},\C)$.  
Using this description it is not difficult to write down a set of equations for
$Y_B$ and prove that, at least when $B$ is smooth, its tangent cone is defined 
by homogeneous polynomials of degree $\ge n$, and the ones of degree $n$ are precisely those in   
the  Yukawa linear system.

It is worth to mention here that the proof of the gauge invariance of the previous construction is a standard consequence of the fact that the period map of the Kuranishi family is induced 
by an $L_{\infty}$ morphism from the Kodaira-Spencer DG-Lie algebra of $X_0$ to an $L_{\infty}$ algebra 
controlling the local structure of the total Grassmannian, giving therefore a constructive and explicit proof that the period map is a morphism of deformation theories. 
In fact, according to general principles pioneered by Nijenhuis (see e.g. \cite{Nij}) and further exposed by
Deligne and Drinfeld in their famous letters \cite{DtM,DtS} (and recently made rigorous by 
several people, especially Lurie~\cite{Lur} and Pridham~\cite{Pri}), 
over a field of characteristic 0 every deformation problem is controlled by a DG-Lie algebra defined up to quasi-isomorphism and 
every morphism of deformation theories is
induced by a morphism in the homotopy category of DG-Lie algebras, which admits a
representative as an $L_{\infty}$ morphism, unique up to homotopy equivalence.

However the situation is not yet completely satisfactory since the above approach does not
give any deformation theoretic interpretation of $Y_B$;
in other words it is not clear under what extent $Y_B$ is the local moduli space
for some deformation problem. The natural starting point is that $Y_B$ fits into a
pull-back diagram
\begin{equation}\label{equ.pullbackdiagram} 
\xymatrix{Y_B\ar[d]\ar[r]&\Grass(F^1H^*(X,\C))\ar[d]\\
B\ar[r]^-{\sP^n}&\Grass(H^*(X,\C))}\end{equation}
which we want to replicate in the category of
$L_{\infty}$ algebras. This is not easy at it seems at a first sight, since the $L_{\infty}$ morphism representing $\sP^n$ is defined up to homotopy and fiber products of $L_{\infty}$-morphisms are  not defined in general (in the category of $L_{\infty}$-algebras). The most natural  solution to both problems is to consider homotopy fiber products. This is possible since DG-Lie algebras 
and $L_{\infty}$-algebras are pointed categories of fibrant objects%
\footnote{As pointed out by one referee, they are precisely the fibrant objects in two model categories, namely the category of DG-Lie algebras equipped with the Quillen model structure and the category of 
cocommutative conilpotent coalgebras equipped with the model structure defined by Hinich \cite{Hin}, respectively.} \cite{kenbrown}, although the usual procedure 
for the construction of homotopy fiber products gives  models which are in general very far from being minimal 
and therefore with Maurer-Cartan equation of difficult geometric interpretation.

The main and probably most difficult  part of this paper is devoted to the construction 
of some small, and in some  cases minimal, models for the homotopy pull-back relative to the 
$L_{\infty}$ version of the above diagram \eqref{equ.pullbackdiagram}.

It is well known that deformations of $X$ are controlled by the Kodaira-Spencer DG-Lie algebra $KS_X$, 
defined as the Dolbeault resolution of the holomorphic tangent sheaf, equipped with the natural bracket and the opposite of the Dolbeault's differential. Let's fix a K\"{a}hler metric on $X$ and denote by $H^{p,q}_X$ the space of harmonic $(p,q)$-forms. If 
\[ U=\oplus_{q}H^{n,q},\qquad V=\oplus_{p<n}\oplus_q H^{p,q},\qquad  W=\oplus_{0<p<n}\oplus_q H^{p,q}\;,\]
then the graded vector spaces $\Hom^*(U,V)[-1]$ and $\Hom^*(U,W)[-1]$, 
considered as  DG-Lie algebras with trivial bracket and trivial differential, control the Grassmann functors of embedded deformations of $F^nH^*(X,\C)$ inside $H^*(X,\C)$ and $F^1H^*(X,\C)$ respectively. 
The first main, and computationally the hardest, result of this paper is the complete 
description of an $L_{\infty}$ morphism 
\[\sP^n\colon KS_X\to \Hom^*(U,V)[-1]\] 
representing the $n$th period map (Theorem~\ref{th:modelperiodmap}); unfortunately 
such a description involves  Green's operators in its Taylor coefficients of order $\ge 2$.  
At this point we can consider the homotopy pull-back diagram 
\[ \xymatrix{\Yuk_X\ar[d]^f\ar[r]&\Hom^*(U,W)[-1]\ar[d]\\
KS_X\ar[r]^-{\sP^n}&\Hom^*(U,V)[-1]}\]
keeping in mind that $\Yuk_X$ is defined up to isomorphism 
in the homotopy category and then its $L_{\infty}$ model is 
determined  only up to quasi-isomorphism.  For simplicity of notation we shall refer to 
$\Yuk_X$ as the Yukawa algebra of $X$.

Our second main result is the description of an $L_{\infty}$  model for the Yukawa algebra in which the 
underlying complex is $KS_X\times \Hom^*(U,V/W)[-2]$ and $f$ is the projection on the first factor
(Theorem~\ref{th:L_ooalgyukawa}): this is done 
by using the theory of derived brackets and $L_{\infty}$ extensions. We shall see that at the level of deformation functors the induced map $f\colon \Def_{\Yuk_X}\to \Def_{KS_X}$ is generally not injective, since its fibre
is pro-represented by the vector space $\Hom(H^{n,0}_X,H^{0,n-1})\oplus \Hom(H^{n,1}_X,H^{0,n})$, while the image 
$f(\Def_{\Yuk_X})\subset \Def_{KS_X}$ is pro-represented by the previously defined locus $Y_B$.
The philosophical interpretation of this fact is that, as frequently happen in deformation theory, the ``geometric'' 
definition
\begin{equation}
Y_U=\left\{u\in U\;\middle|\; \rank\left(F^{n}H^n(X_u,\C)\to \frac{H^n(X_0,\C)}{F^{1}H^n(X_0,\C)}\right)=0\right\}\,.
\end{equation}
requires the additional framing of a homotopy between 0 and the map $F^{n}H^*(X_u,\C)\to \frac{H^*(X_0,\C)}
{F^{1}H^*(X_0,\C)}$ in order to be considered a genuine, i.e., not artificially truncated,   
deformation problem.

The last result of this paper is the proof that if $X$ is a K3 surface, then $\Yuk_X$ is formal; this implies in particular that, denoting by  $\sX\to (B,0)$ its universal deformation, there exists an isomorphism of germs 
$i\colon (B,0)\to (H^1(X,\Theta_X),0)$ such that $i(Y_B)$ is the quadratic cone defined by the Yukawa coupling.

\bigskip
\section{Review of holomorphic Cartan homotopy formulas and K\"{a}hler identities}
\label{sec.CHF}

Unless otherwise specified every vector space is considered over the field of complex numbers. 
Given a complex manifold $X$ let's denote by $(A^{*,*}_X,d=\de+\debar)$
its de Rham complex; later on we shall also consider  its subcomplexes,
\[ A^{\ge p,*}_X=\bigoplus_{i\ge p}A_X^{i,*},\qquad 
A^{>p,*}_X=\bigoplus_{i>p}A_X^{i,*},\]
together the quotient complexes
\[ A^{<p,*}_X=A^{\le p-1,*}_X=\frac{A^{*,*}_X}{A^{\ge p,*}_X},\qquad
 A^{p<*<q,*}_X=\frac{A^{>p,*}_X}{A^{\ge q,*}_X},\qquad A^{i,*}_X=\frac{A^{\ge i,*}_X}{A^{>i,*}_X}\,.\]
Denote by $(A^{0,*}_X(\Theta_X),\debar)$ the Dolbeault complex of the holomorphic
tangent sheaf. 
For notational simplicity, unless otherwise specified,  we shall denote by the same symbol $[-,-]$ both the usual bracket on $A^{0,*}_X(\Theta_X)$ and the graded commutator in
the space $\Hom^*(A^{*,*}_X,A^{*,*}_X)$.
The contraction map $\Theta_X\otimes \Omega_X^p\xrightarrow{\contr}\Omega_X^{p-1}$ extends in the obvious
way to map of degree $-1$:
\[ \bi\colon  A^{0,*}_X(\Theta_X)\to \Hom^*(A^{*,*}_X,A^{*,*}_X),\qquad 
\xi\mapsto \bi_{\xi},\quad \bi_{\xi}(\omega)=\xi\contr \omega\,,\]
which satisfies the ``holomorphic Cartan homotopy formulas'', see e.g. \cite{FMperiods}:
\begin{equation}\label{equ.holomorphiccartanformulas}
 [\bi_{\eta},\bi_{\mu}]=0,\quad
\bi_{\debar\eta}=[\debar,\bi_{\eta}],\quad
\bi_{[\eta,\mu]}=[\bi_{\eta},[\de,\bi_{\mu}]]\,.\end{equation}
The morphism of degree 0
\[ \bl\colon  A^{0,*}_X(\Theta_X)\to \Hom^*(A^{*,*}_X,A^{*,*}_X),\qquad \xi\mapsto\bl_{\xi}=[\de,\bi_{\xi}],
\]
satisfies the equalities
$\bl_{\debar\eta}=-[\debar,\bl_{\eta}]$ and 
$\bl_{[\eta,\mu]}=[\bl_{\eta},\bl_{\mu}]$. A straightforward computation in local coordinates shows that 
every $\xi\in A^{0,p}_X(\Theta_X)$
can be recovered from the linear map $\bl_{\xi}\colon A^{0,0}_X\to A^{0,p}_X$; in particular  both $\bi$ and $\bl$ are injective maps. We shall refer to $\bl_\xi$ as the holomorphic Lie derivative of $\xi$.

For every $\xi\in A^{0,1}_X(\Theta_X)$, the operator $\bi_{\xi}$ is nilpotent of degree $0$ and then it makes sense to consider its exponential
\[ e^{\bi_{\xi}}\colon A^{*,*}_X\to A^{*,*}_X\;.\]
Since $[\bi_{\eta},\bi_{\mu}]=0$ we always have $e^{\bi_{\xi}}e^{\bi_{\eta}}=e^{\bi_{\xi}+\bi_{\eta}}=e^{\bi_{\xi+\eta}}$.

\begin{lemma}\label{lem.integrationequation}
For an element $\xi \in A^{0,1}_X(\Theta_X)$ the following are equivalent:

\begin{enumerate}

\item\label{item1.integrationequation} $\debar\xi=\dfrac{1}{2}[\xi,\xi]$, i.e., $\xi$ is integrable (in the sense of Newlander-Nirenberg);

\item\label{item2.integrationequation} $e^{-\bi_{\xi}}de^{\bi_{\xi}}=d+\bl_{\xi}$;

\item\label{item3.integrationequation} $(d+\bl_{\xi})^2=0$;

\item\label{item4.integrationequation} $(\debar+\bl_{\xi})^2=0$.

\end{enumerate}
\end{lemma}

\begin{proof} The implications $\eqref{item2.integrationequation}\Rightarrow \eqref{item3.integrationequation}\Rightarrow \eqref{item4.integrationequation}$  are completely trivial.
For the equivalence $\eqref{item1.integrationequation}\iff \eqref{item4.integrationequation}$ it is sufficient to write
\[ (\debar+\bl_{\xi})^2=\frac{1}{2}[\debar+\bl_{\xi},\debar+\bl_{\xi}]=[\debar,\bl_{\xi}]+
\frac{1}{2}[\bl_{\xi},\bl_{\xi}]=\bl_{-\debar\xi+\frac{1}{2}[\xi,\xi]}\]
and keep in mind the injectivity of $\bl$. At last, for the implication
$\eqref{item1.integrationequation}\Rightarrow \eqref{item2.integrationequation}$ we have
\[[-\bi_{\xi},d]=[d,\bi_{\xi}]=[\de,\bi_{\xi}]+[\debar,\bi_{\xi}]=\bl_\xi+\bi_{\debar\xi},\qquad
[-\bi_{\xi},\bi_{[\xi,\xi]}]=[-\bi_{\xi},\bi_{\debar\xi}]=0,\]
and therefore
\[ \begin{split}
e^{-\bi_{\xi}}de^{\bi_{\xi}}&=d+\sum_{n\ge 0}\frac{[-\bi_{\xi},-]^n}{(n+1)!}([-\bi_\xi,d])\\
&=d+\bl_\xi+\bi_{\debar\xi}+\sum_{n\ge 1}\frac{[-\bi_{\xi},-]^n}{(n+1)!}(\bl_\xi+\bi_{\debar\xi})\\
&=d+\bl_\xi+\bi_{\debar\xi}-\frac{1}{2}[\bi_{\xi},\bl_\xi]-
\sum_{n\ge 1}\frac{[-\bi_{\xi},-]^n}{(n+2)!}([\bi_{\xi},\bl_\xi])\\
&=d+\bl_\xi+\bi_{\debar\xi}-\frac{1}{2}\bi_{[\xi,\xi]}-
\sum_{n\ge 1}\frac{[-\bi_{\xi},-]^n}{(n+2)!}(\bi_{[\xi,\xi]})\\
&=d+\bl_\xi\,.
\end{split}\]
\end{proof}

In the above setup, let  $n$ be the dimension of $X$. For every $0\le p\le n$ and for every
$\xi\in A^{0,1}_X(\Theta_X)$ denote
\[ F^p_{\xi}=e^{\bi_{\xi}}(A_{X}^{\ge p,*})\subset A^{*,*}_X\;.\]
In particular for $\xi=0$ we recover the usual Hodge filtration of the de Rham complex.
It follows from Lemma~\ref{lem.integrationequation} that
$F^p_{\xi}$ is a subcomplex of $A^{*,*}_X$ whenever $\xi$ is integrable and
\[ e^{\bi_{\xi}}\colon (F^p_0,d+\bl_{\xi})\to (F^p_\xi,d)\]
is an isomorphism of complexes.

\begin{lemma}\label{lem.determinantal1}
Let $\xi,\eta\in A^{0,1}_X(\Theta_X)$ be two integrable sections:
then the image of
\[ H^*(F^n_{\xi},d)\to H^*(X,\C)=H^*(A^{*,*}_X,d)\]
is contained in the image of
\[ H^*(F^1_{\eta},d)\to H^*(X,\C)=H^*(A^{*,*}_X,d)\]
if and only if for every $x\in A^{n,0}_X$ such that $(\debar+\bl_\xi)x=0$ there exists
$y\in A_X^{0,n-1}$ such that
\[\frac{\,\bi_{\xi-\eta}^n\,}{n!}\,x=(\debar+\bl_{\eta})y\;.\]
\end{lemma}

\begin{proof}
Let $e^{\bi_{\xi}}x$ be an element of $F^n_{\xi}$, with $x\in A^{n,*}_X$; we have already
proved that $de^{\bi_{\xi}}x=0$ if and only if $(\debar+\bl_\xi)x=0$.
Then the cohomology class of $e^{\bi_{\xi}}x$ belongs to the image of $H^*(F^1_{\eta},d)\to H^*(X,\C)$ if and only if
there exists  $z\in A^{*,*}_X$
such that $e^{\bi_{\xi}}x-dz\in F^1_{\eta}$ or equivalently if and only if
\[ e^{-\bi_{\eta}}(e^{\bi_{\xi}}x-dz)=e^{\bi_{\xi}-\bi_{\eta}}x-(d+\bl_{\eta})e^{-\bi_{\eta}}z
\in F^1_{0}=A_X^{>0,*}\;.\]
When $x\in A^{n,>0}_X$ the above equation is verified for $z=0$, while for
$x\in A^{n,0}_X$ the above equation admits a solution if and only if
\[\frac{\bi_{\xi-\eta}^n}{n!}\,x=(\debar+\bl_{\eta})y\]
where $y$ is the component  of $e^{-\bi_{\eta}}z$ of type $(0,n-1)$.
\end{proof}

Assume now that $X$ is  a compact K\"{a}hler manifold, i.e.,
a compact complex manifold equipped with K\"{a}hler metric,
then the Laplacians associated to the operators $d,\de,\debar$
satisfy the well known equalities  \cite[p. 44]{Weil}:
\[\Delta_{\debar}=\Delta_{\de}=\frac{1}{2}\Delta_d\]
and therefore they determine the same spaces of harmonic forms
\[ H^{p,q}_X=\ker(\Delta\colon A^{p,q}_X\to A^{p,q}_X),\qquad  \Delta=\Delta_{\debar},\Delta_{\de},\Delta_d\,.\]
Denoting by $\imath\colon H^{p,q}_X \to A^{p,q}_X$ the inclusion, by
$\pi\colon A^{p,q}_X \to H^{p,q}_X$ the harmonic projection and by $G_{\debar}$ the Green operator for
$\Delta_{\debar}$ we have the equalities \cite[p. 66]{Weil}:
\[ \de \imath=\pi\de=\debar\, \imath=\pi\debar=\pi G_{\debar}= G_{\debar}\, \imath=0,\qquad \Delta_{\debar}G_{\debar}=G_{\debar}\Delta_{\debar}=Id-\imath\pi\,,\]
and by Hodge theory the inclusions $\imath\colon (H^{p,*}_X,0) \to (A^{p,*}_X,\debar)$ are quasi-isomorphisms of differential graded vector spaces. Moreover, we have   the following commuting relations in the graded Lie algebra
$\Hom^*(A^{*,*}_X,A^{*,*}_X)$ \cite[pp. 44, 45 and  67]{Weil}:
\[ \Delta_{\debar}=[\debar,\debar^*], \quad
[\de,\debar^*]=[\de,\Delta_\debar]=0,\quad [\de,G_\debar]=[\debar,G_\debar]=[\debar^*,G_\debar]=0\;.\]
Alongside to the excellent Weil's book we also refer to \cite{ManRendiconti,Voisin,Wells} as additional resources
for the above formulas.

\begin{definition}\label{def:propagator} Given a compact K\"{a}hler manifold $X$,  the operator
\[ h=-\debar^*G_{\debar}=-G_{\debar}\debar^*\in \Hom^{-1}(A_X^{*,*},A_X^{*,*})\]
will be called the \emph{$\debar$-propagator}.
\end{definition}

It is straightforward to verify that the $\debar$-propagator satisfies the following identities
\[  [\debar,h]=\debar h+h\debar=\imath\pi-Id,\quad h\imath=\pi h=h^2=0,\quad
[\de,h]=0,\quad [h\de,\debar]=[h,\de\debar]=\de\;.\]

It is worth to point out that the above  equalities give a simple and short  proof of the equality
$\debar\de(A_X^{*,*})=\ker\debar\cap \de(A_X^{*,*})$, and more precisely that a $\de$-exact element $x=\de\alpha$ is $\debar$-closed if and only if $x=\debar\de h\alpha$: in fact  we can write
\[ x=\de \alpha=[h,\de\debar]\alpha= h\de\debar\alpha-\de\debar  h\alpha=
\debar\de h\alpha-h\debar x\;.\]
Similarly the degeneration of the Hodge to de Rham spectral sequence can be proved as a simple consequence of the
formulas $[\de,h\de]=0$, $[h\de,\debar]=\de$, since they easily imply the equality
\[e^{h\de} \debar e^{-h\de}=(Id+h\de)\debar(Id-h\de)=\de+\debar\,.\]

According to the terminology introduced by  Eilenberg and Mac Lane \cite{eilmactw}, cf. also \cite{HK,perturbation},
we may express the equalities
$\debar h+h\debar=\imath\pi-Id$, $h\imath=\pi h=h^2=0$, by saying that for every $k$
the diagram
\[\xymatrix{(H^{k,*}_X,0)\ar@<.4ex>[r]^\imath&(A^{k,*}_X,\debar)\ar@<.4ex>[l]^\pi\ar@(ul,ur)[]^h}\]
is a contraction of complexes.

\bigskip
\section{Infinitesimal deformations and variations of Hodge structures}

According to the standard terminology adopted in deformation theory, by an infinitesimal deformation we mean a deformation over a local Artin ring; when  the Artin ring has square zero  maximal ideal we shall speak of first order deformations.
Let $\Art$ be the category of local Artin $\C$-algebras with residue field $\C$; unless otherwise specified, for every $B\in \Art$ we shall denote by $\mathfrak{m}_B$ its maximal ideal.

Given a complex manifold $X$, let's denote by $\sA^{p,q}_X$ the sheaf of smooth differentiable forms of type $(p,q)$  on $X$.
For a local Artin algebra $B$
and a section $\xi\in A^{0,1}_X(\Theta_X)\otimes\mathfrak{m}_B$, one can consider the sheaf of $B$-modules on $X$:
\[\Oh_{X_{\xi}}=\ker(\debar+\bl_{\xi}\colon  \sA^{0,0}_X\otimes B\to \sA^{0,1}_X\otimes B)\,.\]
It is well known  that $\Oh_{X_{\xi}}$ is a sheaf of flat $B$-modules and
$\Oh_{X_{\xi}}\otimes_B\C=\Oh_X$ if and only if $(\debar+\bl_{\xi})^2=0$. Thus, when $\xi$ is integrable
the local ringed space $X_{\xi}=(X,\Oh_{X_{\xi}})$ is a deformation of $X$ over $\Spec(B)$ and every deformation is obtained, up to isomorphism, in this way. Finally, two integrable sections $\xi,\eta\in A^{0,1}_X(\Theta_X)\otimes\mathfrak{m}_B$ give isomorphic deformations if and only if there exists $a\in A^{0,0}_X(\Theta_X)\otimes\mathfrak{m}_B$ such that $e^{\bl_a}(\debar+\bl_{\xi})=(\debar+\bl_{\eta})e^{\bl_a}$:
for a complete and detailed proof of the above assertions see e.g. \cite{Iaconophd}.

\begin{proposition}\label{prop.periodgeometric} 
In the notation above, for every integrable section $\xi\in A^{0,1}_X(\Theta_X)\otimes\mathfrak{m}_B$ and every positive integer $p$
let $\sF_{\xi}^p\subset \sA^{*,*}_X\otimes B$ be the ideal sheaf
generated by the image of the multiplication map
\[\bigwedge^p d\Oh_{X_{\xi}}\to \bigoplus_{i+j=p} \sA^{i,j}_X\otimes B \,.\]
Then
\[F^p_{\xi}=e^{\bi_\xi}\left(A_X^{\ge p,*}\otimes B\right)=\Gamma(X,\sF_{\xi}^p)\subset A^{*,*}_X\otimes B\;.\]
In particular the subcomplexes $F^p_{\xi}$ represent the variation of the Hodge filtration along the infinitesimal deformation $X_{\xi}$.
\end{proposition}

\begin{proof} This is proved in \cite[Thm. 5.1]{FMperiods}.
It is worth to recall that the filtration $F^p_{\xi}$ depends on the
section $\xi$ and we shall see later that  it can be recovered from the 
deformation $X_{\xi}$ only up to  automorphisms of
the de Rham complex inducing the identity in cohomology.
\end{proof}

When $X$ is compact K\"{a}hler, and $\xi\in A^{0,1}_X(\Theta_X)\otimes\mathfrak{m}_B$ is integrable,
a useful description of the cohomology of the complex $(A^{*,*}_X\otimes B,\debar+\bl_{\xi})$ in terms of the 
$\debar$-propagator
can be obtained by homological perturbation theory. 

\begin{lemma}\label{lem.perturbation}
Let $X$ be compact K\"{a}hler with $\debar$-propagator $h=-\debar^*G_{\debar}$ and denote by 
$H^{*,*}_X\subset A^{*,*}_X$ the subcomplex of harmonic forms. If $B\in \Art$ and 
$\xi\in A^{0,1}_X(\Theta_X)\otimes\mathfrak{m}_B$ is integrable, then the  map
\[ \imath_{\xi}=(Id-h\bl_{\xi})^{-1}\imath=\sum_{n\ge 0}(h\bl_{\xi})^n\imath\colon (H^{*,*}_X\otimes B,0)\to (A^{*,*}_X\otimes B,\debar+\bl_{\xi})\]
is a
homotopy equivalence of complexes of $B$-modules with homotopy inverse
\[ \pi_{\xi}=\pi(Id-\bl_{\xi} h)^{-1}=\sum_{n\ge 0}\pi(\bl_{\xi}h)^n\colon
(A^{*,*}_X\otimes B,\debar+\bl_{\xi})\to (H^{*,*}_X\otimes B,0)\;.\]
In particular the
cohomology groups of $(A^{*,*}_X\otimes B,\debar+\bl_{\xi})$ are free $B$-modules.
\end{lemma}

\begin{proof} In view of the formulas $[\debar,h]=\debar h+h\debar=\imath\pi-Id$, $h\imath=\pi h=h^2=0$, the ordinary perturbation lemma \cite{HK,perturbation} tells that
\[ \imath_{\xi}\colon (H^{*,*}_X\otimes B,\delta_{\xi})\to (A^{*,*}_X\otimes B,\debar+\bl_{\xi}),\qquad
\delta_{\xi}=\sum_{n\ge 0}\pi (\bl_{\xi}h)^n\bl_{\xi}\imath\,,\]
is a homotopy equivalence of complexes of $B$-modules with homotopy inverse $\pi_{\xi}$.
It is now sufficient to observe that,
since $\de\imath=\pi\de=[\de,h]=[\de,\bl_{\xi}]=0$, we have $\bl_{\xi}h\de=\de \bl_{\xi}h$ and then
\[ \delta_{\xi}=\sum_{n\ge 0}\pi (\bl_{\xi}h)^n\bl_{\xi}\imath=
\sum_{n\ge 0}\pi (\bl_{\xi}h)^n\de\bi_{\xi}\imath=
\sum_{n\ge 0}\pi\de (\bl_{\xi}h)^n\bi_{\xi}\imath=0
\,.\]
Notice that $\pi_{\xi}\imath_{\xi}$ is the identity; moreover the ordinary perturbation lemma also says that 
\[ (\debar+\bl_{\xi})h_{\xi}+h_{\xi}(\debar+\bl_{\xi})=\imath_{\xi}\pi_{\xi}-Id,\quad\text{where}\quad
h_{\xi}=\sum_{n\ge 0}h(\bl_{\xi}h)^n\,.\] 
\end{proof}

Since $\de\imath_{\xi}=\sum_{n\ge 0}\de(h\bl_{\xi})^n\imath=\sum_{n\ge 0}(h\bl_{\xi})^n\de\imath=0$, the
Lemma~\ref{lem.perturbation} also gives another proof of the fact that
$H^*(F^{p+1}_{\xi})\to H^*(F^p_{\xi})$ is split injective for every $p$. In fact $(d+\bl_{\xi})\imath_{\xi}=0$
and the conclusion follows by a straightforward chasing in the diagram of complexes
\[ \xymatrix{0\ar[r]&(A^{>p,*}_X\otimes B,d+\bl_{\xi})\ar[r]\ar[d]^{e^{\bi_{\xi}}}&
(A^{\ge p,*}_X\otimes B,d+\bl_{\xi})\ar[r]\ar[d]^{e^{\bi_{\xi}}}&(A^{p,*}_X\otimes B,\debar+\bl_{\xi})\ar[r]
&0\\
&(F^{p+1}_{\xi},d)\ar[r]&(F^{p}_{\xi},d)&(H^{p,*}_X,0)\ar[u]^{\imath_{\xi}}\ar[ul]^{\imath_{\xi}}&}\]
in which the upper row is exact and every vertical arrow is a quasi-isomorphism.

The next proposition  shows that the results of Lemma~\ref{lem.perturbation} hold in a stronger form 
when the de Rham complex is replaced by the subcomplex of $\de$-closed forms.

\begin{proposition}\label{prop.perturbation3}
If $X$ is compact K\"{a}hler and $\xi\in A^{0,1}_X(\Theta_X)\otimes\mathfrak{m}_B$ is integrable, then the  map
\[ (Id-h\bl_{\xi})\colon (\ker\de\otimes B,\debar+\bl_{\xi})\to (\ker\de\otimes B,\debar)\]
is an isomorphism of complexes. In particular, if $\dim X=n$, then the map
\[ \sum_{i\ge 0}(h\bl_{\xi})^i\colon (A^{n,*}_X\otimes B,\debar)\to (A^{n,*}_X\otimes B,\debar+\bl_{\xi})\]
is an isomorphism of complexes.
\end{proposition}

\begin{proof} We first notice that, since $\de h\bl_{\xi}=h\bl_{\xi}\de$ the above maps make sense and
are isomorphisms of graded vector spaces. We have
\[\begin{split}
\debar(Id-h\bl_{\xi})-(Id-h\bl_{\xi})(\debar+\bl_{\xi})&=h\bl_{\xi}\debar+h\bl_{\xi}^2-\debar h\bl_{\xi}-\bl_{\xi}\\
&=h\bl_{\debar\xi}-h\debar\bl_{\xi}+\frac{h}{2}[\bl_{\xi},\bl_{\xi}]-\debar h\bl_{\xi}-\bl_{\xi}\\
&=-\imath\pi\bl_{\xi}\,,
\end{split}\]
and,  whenever $\de x=0$ we get $\imath\pi\bl_{\xi}x=\imath\pi\de\bi_{\xi}x=0$.
\end{proof}

Putting together Lemma~\ref{lem.determinantal1} and Lemma~\ref{lem.perturbation}
we obtain immediately the following result:

\begin{theorem}\label{thm.determinantal2}
Let $X$ be a compact K\"{a}hler manifold of dimension $n$ with $\debar$-propagator $h$ and let
$\xi,\eta\in A^{0,1}_X(\Theta_X)\otimes\mathfrak{m}_B$ be two integrable sections:
then
\[ H^*(F^n_{\xi})\subset H^*(F^1_{\eta})\subset H^*(X,\C)\otimes B\]
if and only if the map
\[ \psi\colon H^{n,0}_X\otimes B\to H^{0,n}_X\otimes B,\qquad
\psi=\sum_{i,j\ge 0}\pi (\bl_{\eta}h)^i\bi_{\xi-\eta}^n
(h\bl_{\xi})^j\,\imath = \pi_\eta \bi_{\xi-\eta}^n\imath_\xi\]
is trivial.
\end{theorem}

\begin{corollary}\label{cor.determinantal2}
Let $X$ be a compact K\"{a}hler manifold of dimension $n$ with $\debar$-propagator $h$ and
$\xi\in A^{0,1}_X(\Theta_X)\otimes\mathfrak{m}_B$ an integrable section:
then
\[ H^*(F^n_{\xi})\subset H^*(A^{>0,*}_X\otimes B)\]
if and only if the map
\[ \psi\colon H^{n,0}_X\otimes B\to H^{0,n}_X\otimes B,\qquad
\psi=\sum_{j\ge 0}\pi \bi_{\xi}^n
(h\bl_{\xi})^j\,\imath = \pi\bi_{\xi}^n\,\imath_\xi \]
is trivial.
\end{corollary}

\bigskip
\section{Review of $A_\infty$ algebras and $L_\infty$ algebras}\label{sec:ooalgebras}

From now on we assume that the reader is familiar with the notion of $A_{\infty}$ and $L_{\infty}$-algebras; this section is devoted to fix notations and recall the basic results that we need in the sequel of this paper.

Given a graded vector space $V=\oplus V^i$ we shall denote by $|x|$ the degree of a homogeneous element; 
given an integer $n$ we shall denote by $V[n]$ the same space with the degrees shifted by $n$, i.e., $V[n]^i=V^{i+n}$ fo every $i$; finally we shall denote by 
$s\colon V[n+1]\to V[n]$  the tautological linear isomorphism of degree $+1$.

For every graded vector space $V$
we shall denote by $\operatorname{Hoch}(V)$ the graded Lie algebra of coderivations of the tensor coalgebra $T(V)=\oplus_{k\geq0}V^{\otimes k}$. Corestriction induces an isomorphism of graded vector spaces 
\[\operatorname{Hoch}(V)\rh\prod_{k\geq0}\Hom^*(V^{\otimes k},V),\quad Q\mapsto pQ=(q_0,q_1,\ldots,q_k,\ldots),\]
where we denote by $p\colon T(V)\rh V$ the natural projection. We call the $q_k\colon V^{\otimes k}\rh V$ the Taylor coefficients of $Q$. We shall denote by $Q^j_k$ the composition $V^{\otimes k}\to T(V)\xrightarrow{Q}T(V)\to V^{\otimes j}$, where the first arrow is the inclusion and the last one is the projection. A coderivation $Q$ is determined by its Taylor coefficients according to $Q^1_0=q_0:V^{\otimes0}=\K\rh V$, $Q^j_0=0$ for $j\neq1$, $Q^0_k=0$ for all $k$, 
\[ Q_k^j(v_1\otimes\cdots\otimes v_k)=\sum_{i=1}^{j}(-1)^{\sum_{h<i}|Q||v_h|}v_1\otimes\cdots\otimes q_{k-j+1}(v_i\otimes\cdots\otimes v_{k+i-j})\otimes\cdots\otimes v_k\] 
for $1\leq j\leq k+1$ and finally $Q_k^j=0$ for $j>k+1$. Given a morphism of graded coaugmented coalgebras $F:T(V)\rh T(W)$ we shall similarly denote by $f_k:V^{\otimes k}\rh W$ the components of the corestriction $T(V)\xrightarrow{F}T(W)\xrightarrow{p}W$ and call them the Taylor coefficients of $F$: the morphism $F$ is determined by its Taylor coefficients according to (where again we denote by $F^j_k$ the composition $V^{\otimes k}\rh T(V)\xrightarrow{F}T(W)\rh W^{\otimes j} $) $F^0_0(1)=1$, $F^j_0=F^0_k=0$ for $j,k\geq1$,
\[ F^j_k(v_1\otimes\cdots\otimes v_k)=\sum_{i_1+\cdots+i_j=k}f_{i_1}(v_1\otimes\cdots\otimes v_{i_1})\otimes\cdots\otimes f_{i_j}(v_{k-i_j+1}\otimes\cdots\otimes v_k)\]
for $1\leq j\leq k$, where the sum is taken over all ordered partitions $i_1+\cdots+i_j=k$ with $j,i_1,\ldots,i_j\geq1$, and finally $F^j_k=0$ for $j>k$.

\begin{definition} An $A_\infty[1]$ algebra structure on a graded space $V$ is a DG-coalgebra structure on $T(V)$ vanishing at $1$, that is,  a degree one coderivation $Q\in\operatorname{Hoch}(V)$ such that $[Q,Q]=0$ and $q_0=0$. An $A_\infty[1]$ morphism between $A_\infty[1]$ algebras $(V,Q)$ and $(W,R)$ is a morphism of DG-coalgebras $F=(f_1\ldots,f_k,\ldots):(V,Q)\rh (W,R)$. A morphism $F$ is called strict if $f_k=0$ for every $k\geq2$. 
\end{definition}

\begin{remark}
It is well known that the above definition is equivalent to the condition  that the higher products on the suspension $V[-1]$, induced by the Taylor coefficients $q_k$ via the inverse d\'ecalage isomorphism $\Hom^*(V^{\otimes k},V)\rh \Hom^*(V[-1]^{\otimes k},V[-1])[1-k]$, define a strong homotopy associative algebra structure, also known as an $A_\infty$ algebra structure, on $V[-1]$. In particular, given a differential graded associative algebra $(A,d,\cdot)$ there is an induced $A_\infty[1]$ algebra structure on $A[1]$ with Taylor coefficients $q_1(s^{-1}a)=-s^{-1}da$, $q_2(s^{-1}a_1\otimes s^{-1}a_2)=(-1)^{|a_1|}s^{-1}(a_1\cdot a_2)$ and $q_k=0$ for $k\geq3$.
\end{remark}

We shall denote by $\CE(V)$ the graded Lie algebra of coderivations of the symmetric coalgebra $S(V)=\oplus_{k\geq0}V^{\odot k}$: again corestriction induces an isomorphism of graded vector spaces 
$\CE(V)\rh\Hom^*(S(V),V)$ and we call the components $q_k\colon V^{\odot k}\rh V$ of $Q\in\CE(V)$ under corestriction the Taylor coefficients of $Q$: they determine $Q$ according to $Q^1_0=q_0$, $Q^j_0=0$ for $j\neq1$, $Q^0_k=0$ for all $k$, 
\[ Q_k^j(v_1\odot\cdots\odot v_k)=\sum_{\sigma\in S(k-j+1,j-1)}\varepsilon(\sigma) q_{k-j+1}(v_{\sigma(1)}\odot\cdots\odot v_{\sigma(k-j+1)})\odot\cdots\odot v_{\sigma(k)},\]
where we denote by $S(p,q)$ the set of $(p,q)$ unshuffles and by $\varepsilon(\sigma)=\varepsilon(\sigma;v_1,\ldots,v_k)$ the Koszul sign. A morphism of graded coaugmented coalgebras $F:S(V)\rh S(W)$ is determined by its Taylor coefficients $f_k:V^{\odot k}\rh W$ according to $F^0_0(1)=1$, $F^j_0=F^0_k=0$ for $j,k\geq1$,
\[ F^j_k(v_1\odot\cdots\odot v_k)=\frac{1}{j!}\sum_{i_1+\cdots+i_j=k}\sum_{\sigma\in S(i_1,\ldots,i_j)}\varepsilon(\sigma)f_{i_1}(v_{\sigma(1)}\odot\cdots\odot v_{\sigma(i_1)})\odot\cdots\odot f_{i_j}(v_{\sigma(k-i_j+1)}\odot\cdots\odot v_{\sigma(k)})\]
for $1\leq j\leq k$ and $F^j_k=0$ for $j>k$.

\begin{definition} An $L_\infty[1]$ algebra structure on a graded vector space $V$ is a DG-coalgebra structure on $S(V)$ vanishing at $1$. An $L_\infty$ morphism between $L_\infty[1]$ algebras is a morphism of DG-coalgebras $F=(f_1\ldots,f_k,\ldots):(V,Q)\rh (W,R)$: it is called strict if $f_k=0$ for every $k\geq2$. 
\end{definition}

\begin{remark}
Again, the above definition is equivalent to say that the higher brackets on $V[-1]$ induced by the inverse d\'ecalage isomorphism $\Hom^*(V^{\odot k},V)\rh \Hom^*(V[-1]^{\wedge k},V[-1])[1-k]$ define a strong homotopy Lie algebra structure, also known as an $L_\infty$ algebra structure, on $V[-1]$: cf. \cite{LadaMarkl}. In particular, given a 
DG-Lie algebra $(L,d,[\cdot,\cdot])$ there is an induced $L_\infty[1]$ algebra structure on $L[1]$ with Taylor coefficients $q_1(s^{-1}l)=-s^{-1}dl$, $q_2(s^{-1}l_1\odot s^{-1}l_2)=(-1)^{|l_1|}s^{-1}[l_1,l_2]$ and $q_k=0$ for $k\geq3$.
\end{remark}

\begin{remark}\label{rem:symmetrization} There is a symmetrization functor $\operatorname{sym}\colon\mathbf{A_\infty}[1]\rh\mathbf{L_\infty}[1]$ from the category 
of $A_\infty[1]$ algebras to that of $L_\infty[1]$ algebras \cite{LadaMarkl}, generalizing the classical construction 
from the category of differential graded associative algebras to the category of differential graded Lie algebras: 
an $A_\infty[1]$ algebra $(V,q_1,\ldots,q_k,\ldots)$ is mapped to $(V,\operatorname{sym}(q_1),\ldots,\operatorname{sym}(q_k),\ldots)$,
\[\operatorname{sym}(q_k)(v_1\odot\cdots\odot v_k)=\sum_{\sigma\in S_k}\varepsilon(\sigma) q_k(v_{\sigma(1)}\otimes\cdots\otimes v_{\sigma(k)}),\]
and sending an $A_\infty[1]$ morphism $F=(f_1,\ldots,f_k,\ldots):(V,Q)\rh (W,R)$ to the $L_\infty[1]$ morphism $\operatorname{sym}(F)=(\operatorname{sym}(f_1),\ldots,\operatorname{sym}(f_k),\ldots):(V,\operatorname{sym}(Q))\rh(W,\operatorname{sym}(R))$,
\[\operatorname{sym}(f_k)(v_1\odot\cdots\odot v_k)=\sum_{\sigma\in S_k}\varepsilon(\sigma) f_k(v_{\sigma(1)}\otimes\cdots\otimes v_{\sigma(k)}).\]         \end{remark}

\begin{definition}
Given an $A_\infty[1]$ (resp.: $L_\infty[1]$) morphism $F=(f_1,\ldots,f_k,\ldots)\colon (V,Q)\to (W,R)$ then $f_1;(V,q_1)\to (W,r_1)$ is a morphism of complexes: we shall say that $F$ is a quasi-isomorphism, or equivalently a weak equivalence, if such is $f_1$.
\end{definition}

Most of our computations will be based on the following important and well known homotopy transfer theorem: for a proof we refer to \cite{FMcone,Hue2010} and references therein. 

\begin{theorem} Given an $A_\infty[1]$ (resp.: $L_\infty[1]$) algebra $(V,q_1,\ldots,q_k,\ldots)$, a differential graded vector space $(W,r_1)$ and a diagram 
\[ \xymatrix{(W, r_1)\ar@<.4ex>[r]^-{f_1}&(V, q_1)\ar@<.4ex>[l]^-{g_1}\ar@(ul,ur)[]^K} \]
where $f_1,g_1$ are morphisms of cochain complexes and $K$ is a homotopy such that 
$g_1f_1=Id$, $q_1K+Kq_1=f_1g_1-Id$, then 
there is an induced $A_\infty[1]$ (resp.: $L_\infty[1]$) structure $(W,r_1,\ldots,r_k,\ldots)$ on $W$ and an $A_\infty[1]$ (resp.: $L_\infty[1]$) quasi-isomorphism $F=(f_1,\ldots,f_k,\ldots)\colon W\to V$, defined by the recursions
\[ f_k=\sum_{j=2}^k Kq_jF^j_k,\quad\text{for}\quad k\geq2,\] 
\[ r_k =\sum_{j=2}^k g_1q_jF^j_k,\quad\text{for}\quad k\geq2.\]
Notice that $F^{j}_k$ only depends on $f_1,\ldots,f_{k-j+1}$. Moreover there exists  
an $A_\infty[1]$ (resp.: $L_\infty[1]$) quasi-isomorphism $G=(g_1,\ldots,g_k,\ldots)\colon V\to W$ such that $GF$ is the identity.
In the $A_\infty[1]$ case the Taylor coefficients of the morphism $G$ may be defined by the recursive formulas
\[ g_k = \sum_{j=1}^{k-1}g_jQ_k^{j}K_k,\]
where $K_k:=\sum_{i=0}^{k-1}\id^{\otimes i}\otimes K\otimes (f_1g_1)^{k-i+1}\colon V^{\otimes k}\to V^{\otimes k}$.
\end{theorem}

\begin{remark}\label{rem:sym} It is useful to notice that 
homotopy transfer is compatible in the obvious sense with the symmetrization functor $\operatorname{sym}:\mathbf{A_\infty}[1]\rh\mathbf{L_\infty}[1]$ of Remark~\ref{rem:symmetrization}. 
In the $L_\infty[1]$ case it is also possible to give an explicit description of the quasi-isomorphism $G$, although via recursive formulas more complicated than in the  $A_\infty[1]$ case, cf. \cite{berglund,prelieDT}.
\end{remark}

We have already defined the category $\Art$  of local Artin $\C$-algebras with residue field $\C$; in addition we shall denote by $\Set$ the category of sets and  by 
$\mathbf{Grp}$ the category of groups. 

For a Lie algebra $L^0$, its exponential functor is 
\[ \exp_{L^0}\colon \Art\to \mathbf{Grp},\qquad \exp_{L^0}(B)=\exp(L^0\otimes\mathfrak{m}_B)\;.\]

Given a DG-Lie algebra $L$, its Maurer-Cartan functor is
\[ \MC_L\colon \Art\to\Set,\qquad \MC_L(B)=\{x\in L^1\otimes\mathfrak{m}_B\mid dx+\frac{1}{2}[x,x]=0\}.\]
The (left) gauge action of $\exp_{L^0}$ on $\MC_L$ may be written as
\[ e^a\ast x=x+\sum_{n\ge 0}\frac{[a,-]^n}{(n+1)!}([a,x]-da)\]
and the corresponding quotient is called the deformation functor associated to $L$: 
\[ \Def_L=\frac{\MC_L}{\;\exp_{L^0}\;}.\]

The basic theorem of deformation theory asserts that every quasi-isomorphism $L\to M$ of DG-Lie algebras induces an isomorphism of functors $\Def_L\to \Def_M$.
In one of its simplest interpretations, the basic theorem of derived deformation theory asserts that in characteristic 0 every  ``deformation problem'' is described by the functor $\Def_L$ for a suitable DG-Lie algebra $L$ 
determined up to homotopy; similarly every ``morphism of deformation theories'' is induced by a morphism in the homotopy category of DG-Lie algebras: here enters in the game also $L_{\infty}$-algebras as a fundamental tool, 
every morphism in the homotopy category of DG-Lie algebras can be represented by a direct  
$L_{\infty}$ morphism.   

The construction of the functor $\Def_L$ extends, although in a non trivial way, also to $L_{\infty}[1]$ algebras.
First notice that if $V=(V,q_1,q_2,\ldots)$ is an 
$L_{\infty}[1]$ algebra and $(C,d)$ is a commutative DG-algebra, then the  natural extensions   
\[ q_1\colon V\otimes C\to V\otimes C,\qquad q_1(v\otimes c)=q_1(v)\otimes c+(-1)^{|v|}v\otimes dc,\]
\[ q_2\colon (V\otimes C)^{\odot 2}\to V\otimes C,\qquad q_2(v_1\otimes c_1,v_2\otimes c_2)=
(-1)^{|c_1|\,|v_2|}q_2(v_1,v_2)\otimes c_1c_2,\]
\[ q_k\colon (V\otimes C)^{\odot k}\to V\otimes C,\qquad q_k(\odot_{i=1}^kv_i\otimes c_i)=
\pm q_k(v_1\odot\cdots\odot v_k)\otimes c_1\cdots c_k,\]
gives an $L_{\infty}[1]$ algebra structure on $V\otimes C$.
Then the Maurer-Cartan functor is defined as
\[ \MC_V\colon \Art\to \Set,\qquad   \MC_V(B)=\left\{x\in V^0\otimes\mathfrak{m}_B\;\middle|\; \sum_{n>0}\frac{1}{n!}q_n(x^{\odot n})\right\}\;,\]
while the gauge action is replaced by the homotopy equivalence: two Maurer-Cartan elements $x,y\in \MC_V(B)$ are said to be homotopy equivalent if there exists $z(t)\in \MC_{V[t,dt]}(B)$ such that $z(0)=x$ and $z(1)=y$;
here $V[t,dt]:=V\otimes \C[t,dt]$  and $\C[t,dt]$ is  the de DG-algebra of polynomial differential forms on the affine line.
The functor $\Def_V$ is defined as the quotient of $\MC_V$ by homotopy equivalence.

For a DG-Lie algebra $L$ the above definition of $\Def_L$ coincides with the classical one. We also remark that every $L_{\infty}$-morphism $V\to W$ induces natural transformations $\MC_V\to \MC_W$, $\Def_V\to\Def_W$, cf. \cite{ManRendiconti,SS1}.

\bigskip
\section{The Fiorenza-Manetti model for the local period maps}
\label{sec:FM}

The starting point of this paper is given by the results of \cite{FMperiods}, some of them we have already used in 
Proposition~\ref{prop.periodgeometric}. Apart from the technicalities, the main contribution of \cite{FMperiods} 
is the fact that the ``correct'' period domain for a compact K\"{a}hler manifold $X$ is not a subset of the product of 
the Grassmannians $\prod_i\Grass(H^i(X,\C))$ but the Grassmannian of the entire de Rham complex $(A^{*,*}_X,d)$, 
which is a completely different object in the framework of derived geometry and carries a richer algebro-geometric structure, cf. \cite{Carmelo1,Carmelo2}.

Given a  vector space $V$ and a vector subspace $F\subset V$, the Grassmann functor of the pair $(V,F)$ is 
\[ G_{V,F}\colon \Art\to \Set,\quad 
G_{V,F}(B)=\{\sF\mid \sF\subset V\otimes B\; \text{flat $B$-submodule},\; 
\sF\otimes_B\K=F\}.\]
If $V$ is finite dimensional,   the tautological bundle on the Grassmannian is a universal family for the above functor and then 
$G_{V,F}(B)$ is identified with the set of morphism of pointed schemes $(\Spec(B),0)\to (\Grass(V),F)$.

If $V$ is a complex of vector spaces and $F\subset V$ is a subcomplex we can define the functor $G_{V,F}$ as 
\[ \operatorname{G}_{V,F}(B)=\frac{\left\{\begin{array}{l}
\text{subcomplexes of flat $B$-modules }
\text{$\sF\subset V\otimes B$ such that $\sF\otimes_B\K=F$} \end{array}\right\}}
{\left\{\begin{array}{l}
\text{$B$-linear automorphisms of the complex $V\otimes B$ lifting the identity }\\
\text{on $V$  and  inducing the identity in cohomology} \end{array}\right\}}
\]

The main reason of taking quotient for the automorphisms which are homotopy equivalent to the identity is that in this way 
the functor $G_{V,F}$ is homotopy invariant; this means that 
every morphism of pairs 
$\alpha\colon (V,F)\to (W,H)$ such that both $\alpha\colon V\to W$ and $\alpha\colon F\to H$ are quasi-isomorphisms, induces an isomorphism of functors $G_{V,F}\simeq G_{W,H}$. This fact is essentially  proved in \cite{FMperiods} (especially in the ArXiv versions); a more detailed study will appear in the forthcoming paper \cite{FMAJ}.
A first consequence of the homotopy invariance is that if $H^*(F)\to H^*(V)$ is injective, then the natural transformation 
\[ G_{V,F}\to G_{H^*(V),H^*(F)},\qquad \sF\mapsto H^*(\sF),\]
is well defined and is an isomorphism of functors: to see this it is sufficient to choose a set  $H\subset V$ of harmonic representatives for the cohomology of $V$ such that $H\cap F\to F$ is a quasi-isomorphism and apply homotopy invariance to the morphism of pairs $(H,H\cap F)\to (V,F)$.

There exist several equivalent models of  $L_{\infty}$ algebras governing the functor $G_{V,F}$; in this section we explain the one  usually called Fiorenza-Manetti mapping cone, which is very convenient for the algebraic 
description of period maps. Another model, much more convenient for the goal of this paper, will be described in next sections. To this end we need to start by recalling some basic stuff about homotopy fibers of morphisms of differential graded Lie algebras.

Given a morphism of DG-Lie algebras 
$f\colon L\to M$ one can define the analog of Maurer-Cartan functor: 
\[ \MC_f(A)=\left\{(l,e^m)\in L^1\otimes\mathfrak{m}_A\times \exp(M^0\otimes\mathfrak{m}_A)\;\middle|\; 
dl+\frac{1}{2}[l,l]=0,\; e^m\ast f(l)=0\right\}.\]
The gauge action on $\MC_L$ lifts to a (left) action of the group functor $\exp_{L^0}\times\exp_{dM^{-1}}$ 
on $\MC_f$ by setting
\[ (e^a,e^{du})\ast (l,e^m)=(e^a\ast l,\, e^{du}e^me^{-f(a)}),\qquad (l,u)\in (L^0\oplus M^{-1})\otimes\mathfrak{m}_A.\]
 
 \begin{lemma}[{\cite[Thm. 6.14]{ManettiSemireg}}]\label{lem.modellopiccolofibraomotopica}
Given a 
morphism of differential graded Lie algebras $f\colon L\to M$,  the functor of Artin rings  
\[ \Def_f=\frac{\MC_f}{\;\exp_{L^0}\times\exp_{dM^{-1}}\;}\;.\]
is  isomorphic to the deformation functor of the homotopy fiber of $f$. This means that
every commutative diagram of differential graded Lie algebras
\[ \xymatrix{L\ar[d]^f\ar[r]^\alpha&P\ar[d]^q&\ker(q)\ar[l]_i\ar[d]^h\\
M\ar[r]^{\beta}&Q&0\ar[l]}\]
with $q$ surjective, $\alpha,\beta$ quasi-isomorphisms and $i$ the inclusion, 
induces canonically two isomorphisms of functors
\[ \Def_f\xrightarrow{\;\cong\;}\Def_q\xleftarrow{\;\cong\;}\Def_h=\Def_{\ker(q)}\;.\]
\end{lemma}

Notice that, when $f\colon L\to M$ is an injective morphism of DG-Lie algebras, the Maurer-Cartan functor admits the simpler description 
\[ \MC_f(A)=\{e^m\in \exp(M^0\otimes\mathfrak{m}_A)\mid 
e^{-m}\ast 0\in f(L^1)\otimes \mathfrak{m}_A\}.\]
On the other side, when $M=0$ the functors $\MC_f$ and $\Def_f$ reduces to the usual definitions of $\MC_L$ and $\Def_L$ respectively.

\begin{remark} The importance of homotopy fibres in deformation theory has been clarified in several places, see e.g. \cite{FMcone,FMAJ,ManettiSemireg}, since they are the right  object
for the study of  semitrivialized deformation problems.

We shall talk about semitrivialized deformation problems when
we consider deformations of a geometric object together with a trivialization of
the deformation of a specific part of it. For instance, in this class we have Grassmann functors and more generally  
embedded deformations
of a subvariety $Z$ of a complex manifold $X$: such deformations (over a base $B$) can
be considered as deformations  $\mathcal{Z}\subset \mathcal{X}$  of the pairs $Z\subset X$ equipped with  a trivialization $\mathcal{X}\simeq X\times B$.

Keeping in mind that \emph{deformations = solutions of Maurer-Cartan
equation} and \emph{trivializations = gauge equivalences}, it becomes natural 
that, according to Lemma~\ref{lem.modellopiccolofibraomotopica}, every
semitrivialized deformation problem is governed by a functor $\Def_f$ for a suitable 
morphism of differential graded Lie algebras $f\colon L\to M$.
\end{remark}

\begin{example}\label{ex.grassmannianecomplessi}

Let $(V,d)$ be a complex of  vector spaces and $F\subset V$ a subcomplex. Then we have an inclusion morphism of DG-Lie algebras $\chi\colon \End^*(V;F)\to \End^*(V)=\Hom^*(V,V)$, where
\[  \End^*(V;F)=\{f\in \End^*(V)\mid f(F)\subset F\}.\]
Given $A\in \Art$ and $m\in \Hom^0(V,V)\otimes\mathfrak{m}_A$, from the formula 
$e^{-m}de^{m}=d+e^{-m}\ast 0\,$ it follows immediately that that 
$e^m\in \MC_{\chi}(A)$ if and only if $e^{m}(F\otimes A)$ is a subcomplex of $V\otimes A$.
If $e^m,e^n\in \MC_{\chi}(A)$ are gauge equivalent, i.e., if 
$e^n=e^{du}e^me^{-\chi(a)}$ then $e^{du}\colon V\otimes A\to V\otimes A$ is a morphism of complexes homotopic to the identity and, since $e^{n}(F\otimes A)=e^{du}e^{m}(F\otimes A)$,  the two maps 
\[  H^*(e^{n}(F\otimes A))\to H^*(V\otimes A),\qquad H^*(e^{m}(F\otimes A))\to H^*(V\otimes A),\]
have the same image. 
Thus we have defined a natural transformation 
\[ \Def_{\chi}\to G_{V,F}\]
which, by the results of \cite{FMperiods} is an isomorphism of functors. 
\end{example}

One of the main results of \cite{FMcone} is the concrete description of an 
$L_{\infty}$-algebra  $C(f)$ representing the homotopy fiber of $f\colon L\to M$, together with a natural 
isomorphism of functors $\MC_f\simeq \MC_{C(f)}$.  The underlying complex is the 
mapping cocone of $f$ in the category of complexes, i.e.,  $C(f)=L\times M[-1]$ and therefore $C(f)[1]=L[1]\times M$
equipped with the differential 
\[q_1(s^{-1}l,m)=(-s^{-1}dl,dm-f(l))\,.\]
The only non trivial contributions to the higher brackets $q_k\colon C(f)[1]^{\odot k}\to C(f)[1]$, $k>1$, are
\[ q_2\colon L[1]^{\odot 2}\to L[1],\qquad q_{k+1}\colon L[1]\otimes M^{\odot k}\to M,\quad k\ge 1,\]
defined by the formulas
\begin{equation}\label{brcocone}
\begin{split}
q_2(s^{-1}l_1\odot s^{-1}l_2)&=(-1)^{|l_1|}s^{-1}[l_1,l_2],\\
q_{k+1}(s^{-1}l\otimes m_1\odot\cdots\odot m_k)&=-\frac{B_k}{k!}\sum_{\sigma\in S_k}\varepsilon(\sigma)[\cdots[[f(l),m_{\sigma(1)}],m_{\sigma(2)}]\cdots,m_{\sigma(k)}],\end{split}
\end{equation}
where $\varepsilon(\sigma)$ is the Koszul sign and 
$B_0,B_1,\ldots$ are the Bernoulli numbers:
\[ \frac{t}{e^t-1}=\sum_{k\ge 0}\frac{B_k}{k!}t^k=1-\frac{t}{2}+\frac{t^2}{12}-\frac{t^4}{720}+\frac{t^6}{30240}-
\frac{t^8}{1209600}+\cdots\;.\]

\begin{lemma}[{\cite[Thm. 7.5]{FMcone}}]\label{lem.conevshofib} 
In the above notation, $C(f)$ is weak equivalent (as $L_{\infty}$-algebra) to  the 
homotopy fibre of $f\colon L\to M$. 
For every $B\in \Art$ we have 
\begin{equation}\label{equ.ledueMC}
(x,m)\in \MC_{C(f)}(B)\iff (x,e^m)\in \MC_f(B)\;.\end{equation} 
Moreover  
$(x_1,m_1),(x_2,m_2)\in \MC_{C(f)}(B)$ are homotopy equivalent if and only if 
$(x_1,e^{m_1}),(x_2,e^{m_2})\in \MC_{f}(B)$ are gauge equivalent; in other words the isomorphism \eqref{equ.ledueMC} 
induces an isomorphism of functors $\Def_{C(f)}=\Def_f$.
\end{lemma}


For a complex manifold $X$, the functor $\Def_X\colon \Art\to \Set$ of its infinitesimal 
deformations is controlled by the Kodaira-Spencer DG-Lie algebra 
\[ KS_X=(A^{0,*}_X(\Theta_X),-\debar,[-,-])\;,\]
while the deformations of its  de Rham complex  are controlled by the DG-Lie algebra
$\End^*(A^{*,*}_X)=\Hom^*(A^{*,*}_X,A^{*,*}_X)$ and
the Hodge filtration gives a sequence of inclusions of  DG-Lie subalgebra  
\begin{equation}\label{eq:chi_p}\End^*(A^{*,*}_X; A^{\ge p,*}_X )=\{f\in \End^*(A^{*,*}_X)\mid f(A^{\ge p,*}_X)\subseteq A^{\ge p,*}_X\}
\xrightarrow{\chi_p}\End^*(A^{*,*}_X)\;.\end{equation}
According to Example~\ref{ex.grassmannianecomplessi} and  Lemma~\ref{lem.conevshofib}, the $L_{\infty}$ structure on 
$C(\chi_{i})$ controls the homotopical deformations of the $i$th subcomplex of the Hodge filtration inside
the de Rham complex: here homotopical means that two deformations are considered isomorphic if they differ by an automorphism of the de Rham complex which is homotopic to the identity.    

It is therefore natural to expect that the universal $i$th period map is induced by an $L_{\infty}$-morphism from 
$KS_X$ to $C(\chi_i)$; in fact we have:

\begin{theorem}\label{th:FMmodelperiod} In the above notation, the map 
\[ \sP^i\colon KS_X\to C(\chi_i),\qquad \xi\mapsto (\bl_{\xi},\bi_{\xi}),\]
is a strict $L_{\infty}$-morphism inducing the $i$th period map.
\end{theorem}

\begin{proof} For the proof we refer to \cite{FMperiods}. However, the proof that $\sP^i$ is a strict 
$L_{\infty}$-morphism is a formal consequence of the Cartan homotopy formulas and will be reproduced later in a more abstract setting. 
Notice that for $B\in \Art$, we have 
\[ \MC_{KS_X}(B)=\left\{\xi\in A^{0,1}_X(\Theta_X)\otimes\mathfrak{m}_B\;\middle|\; 
\debar\xi=\frac{1}{2}[\xi,\xi]\right\}.\]
Given an integrable section $\xi\in \MC_{KS_X}(B)$ its image under $\sP^i$ is the Maurer-Cartan element 
$(\bl_{\xi},\bi_{\xi})\in \MC_{C(f)}(B)$ which correspond to the 
deformed subcomplex $e^{\bi_{\xi}}(A^{\ge i,*}_X\otimes B)$. 
\end{proof}

\bigskip
\section{Models of the homotopy fiber}
\label{sec:homfiber}

The aim of this section is to compare several $A_\infty$ models of the homotopy fiber of an inclusion of DG associative algebras: we will apply these results in the next section to the case of the inclusion $\chi_p$ from \eqref{eq:chi_p}. 
As recalled in Section \ref{sec:FM}, for any morphism of DG Lie algebras 
$f\colon L\to M$ we may put an $L_\infty$ algebra structure on the mapping cocone $\cocone(f)=L\times M[-1]$ which is a model for the homotopy fiber $0\times_M^h L$ of $f$. This $L_\infty$ structure is induced via homotopy transfer from the following DG Lie algebra model for $0\times_M^h L$
\[ K_f:=\{ (l,m(t,dt))\in L\times M[t,dt]\,\,\mbox{s.t.}\,\, m(t,dt)_{|t=0}=0,\, m(t,dt)_{|t=1}=f(l) \},\]
along a certain natural contraction called Dupont's contraction, cf. \cite{chengGetzler,FMcone}; 
we also refer to \cite{donatella} for the explicit 
$L_\infty$ algebra structure on the mapping cocone $\cocone(f-g)$ which is a model for the homotopy equalizer of two morphism $f,g\colon L\to M$ of DG-Lie algebras.

For a morphism $f:A\rh B$ of DG associative algebras we may similarly put a DG associative algebra structure on $K_f$ and via homotopy transfer along Dupont's contraction an $A_\infty$ algebra structure on the mapping cocone $\cocone(f)=A\times B[-1]$. This $A_\infty$ structure was essentially computed in \cite[Prop. 19]{chengGetzler}: the differential is as usual 
$q_1(s^{-1}a,b)=(-s^{-1}da,db-f(a))$, the nontrivial contributions to the higher products $q_k\colon\cocone(f)[1]^{\otimes k}\to\cocone(f)[1]$ are $q_2(s^{-1}a_1\otimes s^{-1}a_2)=(-1)^{|a_1|}s^{-1}(a_1a_2)$ and
\begin{multline}\label{prcocone} q_{i+j+1}(b_1\otimes\cdots\otimes b_i\otimes s^{-1}a\otimes b_{i+1}\otimes \cdots \otimes b_{i+j} )=\\ 
=\frac{B_{i+j}}{i!\,j!}(-1)^{i+1+\sum_{h=1}^{i}|b_h|}\;b_1\cdots b_i f(a) b_{i+1}\cdots b_{i+j},\qquad \forall\,\, i+j\geq1.\end{multline}
It is clear, either by Remark~\ref{rem:sym} or by a direct computation (recall the well known formula $[a,-]^k(b)=\sum_{i=0}^k(-1)^i\binom{k}{i}a^{k-i}ba^{i}$ and apply polarization, also recall $B_{i+j}=(-1)^{i+j}B_{i+j}$ for $i+j\geq2$) that the symmetrized $L_\infty$ algebra structure on $C(f)$ coincides with Fiorenza-Manetti mapping cocone of 
$f_{Lie}\colon A_{Lie}\rh B_{Lie}$.

In the DG associative case there is however a simpler DG associative algebra structure on $\cocone(f)$ given by the cup product 
$\cup\colon\cocone(f)^{\otimes2}\to\cocone(f)$,
\begin{equation}\label{asscocone} (a_1,sb_1)\cup(a_2,sb_2)=(a_1a_2,s(b_1f(a_2))), \end{equation}
which is also immediately seen to be a model for the homotopy fiber $0\times^h_BA$. The first result of this section is the determination of explicit $A_\infty$ isomorphisms between these two models of the homotopy fiber, roughly given by the exponential and logarithm on $B$: this will be used to simplify the computations in Theorem~\ref{thm.derivedproducts}.

\begin{definition} We will call $\cocone(f)$ with the DG algebra structure \eqref{asscocone} the associative mapping cocone of $f$ and denote it by $\cocone(f)_{As}$, while we will call $\cocone(f)$ with the $A_\infty$ algebra structure \eqref{prcocone} the Fiorenza-Manetti mapping cocone of $f$ and denote it by $\cocone(f)_\infty$. \end{definition}

\begin{lemma}\label{lem.assvsinftycocone} 
Given a morphism $f\colon A\to B$ of DG associative algebras, there exist two  isomorphisms of $A_\infty[1]$ algebras
$E\colon\cocone(f)_\infty[1]\to\cocone(f)_{As}[1]$ and $L\colon\cocone(f)_{As}[1]\to\cocone(f)_{\infty}[1]$, the one inverse of the other, whose 
Taylor coefficients $e_k,l_k\colon\cocone(f)[1]^{\otimes k}\to\cocone(f)[1]$ are 
$e_1=l_1=\id_{\cocone(f)[1]}$ and, for $k\geq2$,
\[e_k((s^{-1}a_1,b_1)\otimes\cdots\otimes(s^{-1}a_k,b_k))=\left(0,\frac{1}{k!}b_1\cdots b_k\right),\]
\[l_k((s^{-1}a_1,b_1)\otimes\cdots\otimes(s^{-1}a_k,b_k))=\left(0,\frac{(-1)^{k+1}}{k}b_1\cdots b_k\right).\]
\end{lemma}

\begin{proof} We denote by $(\cocone(f)_{As}[1],r_1,r_2,0,\ldots,0,\ldots)$ the $A_\infty[1]$ structure on $\cocone(f)_{As}[1]$. It is straightforward that $E$ and $L$ are inverse automorphisms of the tensor coalgebra over $\cocone(f)[1]$, thus it suffices to show that $L$ is an $A_\infty[1]$ morphism, that is, the relation
\begin{equation}\label{eqn}l_k R^k_k + l_{k-1}R^{k-1}_k = \sum_{i=1}^k q_iL^i_k,\qquad k\geq2,\end{equation}
(cf. Section \ref{sec:ooalgebras} for notations). Looking at the explicit formulas \eqref{prcocone} and \eqref{asscocone} one quickly realizes that the only cases where the necessary relation \eqref{eqn} is not trivially satisfied are terms of type $b_1\otimes\cdots\otimes b_i\otimes s^{-1}a\otimes b_{i+1}\otimes\cdots\otimes b_{i+j}$, $i+j+1=k$: in this case the left hand side of \eqref{eqn} is easily computed and it is equal to
\begin{equation}\label{eqn2}(l_k R^k_k + l_{k-1}R^{k-1}_k)(s^{-1}a\otimes b_{1}\otimes\cdots\otimes b_{j}) = \frac{(-1)^{j+1}}{j+1}f(a)b_1\cdots b_j\end{equation}
when $i=0$, and to
\begin{multline}\label{eqn3}
(l_k R^k_k + l_{k-1}R^{k-1}_k)(b_1\otimes\cdots\otimes b_i\otimes s^{-1}a\otimes b_{i+1}\otimes\cdots\otimes b_{i+j}) =\\
= (-1)^{\sum_{h=1}^{i}|b_h|}\left(\frac{(-1)^{i+j+1}}{i+j+1}+\frac{(-1)^{i+j}}{i+j}\right)b_1\cdots 
b_if(a)b_{i+1}\cdots b_{i+j}\end{multline}
when $i\geq1$. We turn our attention to the rhs of \eqref{eqn}: this becomes
\[\sum_{i_1+\cdots+i_p=i}\sum_{j_1+\cdots+j_q=j}q_{p+q+1}\left(\frac{(-1)^{i_1+1}}{i_1}b_1\cdots b_{i_1}\otimes\cdots\otimes s^{-1}a\otimes\cdots\otimes\frac{(-1)^{j_q+1}}{j_q}b_{i+j-j_q+1}\cdots b_{i+j}\right).\]
By the formula \eqref{prcocone} it's clear that this is equal to $(-1)^{\sum_{h=1}^{i}|b_h|}C_{i,j}b_1\cdots b_if(a)b_{i+1}\cdots b_{i+j}$ for some coefficient $C_{i,j}\in\mathbb{Q}$, and a little more thought shows that we can identify $C_{i,j}$ with the coefficient of $z^iw^j$ in the Taylor series expansion of 
\[\varphi(x,y)=\sum_{i,j\geq0}\frac{(-1)^{i+1}B_{i+j}}{i!j!}x^i y^j= -\sum_{k\geq0}\frac{B_k}{k!}(y-x)^k=-\frac{y-x}{e^{y-x}-1}\] 
after the substitution 
\[x=\sum_{k\geq1}\frac{(-1)^{k+1}}{k}z^k=\log(1+z),\qquad y=\sum_{k\geq1}\frac{(-1)^{k+1}}{k}w^k=\log(1+w)\,.\] 
Next we observe that $e^{\log(1+w)-\log(1+z)}-1=\dfrac{1+w}{1+z}-1=\dfrac{w-z}{1+z}$, and therefore
\[\begin{split}
\varphi(\log(1+z),\log(1+w))&=-\frac{1+z}{w-z}(\log(1+w)-\log(1+z))\\
&=\frac{1+z}{w-z}\sum_{k\geq1}\frac{(-1)^{k}}{k}(w-z)(w^{k-1}+zw^{k-2}+\cdots+z^{k-1}),\end{split}\]
from which we get $C_{0,j}=\dfrac{(-1)^{j+1}}{j+1}$ and 
$C_{i,j}=\dfrac{(-1)^{i+j+1}}{i+j+1}+\dfrac{(-1)^{i+j}}{i+j}$ for $i\geq1$. 
By comparison with \eqref{eqn2} and \eqref{eqn3} this concludes the proof of the lemma.
\end{proof}

Next we consider the particular case when $f=i\colon A\to B$ is an inclusion. Given a graded subspace $C\subset B$ such that $B=A\oplus C$ as graded spaces we can consider the following contraction, where we denote by $P\colon B\to C$ the projection with kernel $A$, by $P^{\bot}=\id_B-P$, by $d$ the differential on $B$ and by $q_1$ the differential on $\cocone(i)[1]$ (notice that since $d(A)\subset A$ we have $(Pd)^2=P(d)^2=0$ and thus $Pd$ is a differential on $C$):
\begin{equation}\label{contr}
\begin{gathered}
\xymatrix{(C, Pd)\ar@<.4ex>[r]^-{f_1}&(\cocone(i)[1], q_1)\ar@<.4ex>[l]^-{g_1}\ar@(ul,ur)[]^K},\\ 
f_1(c)=(s^{-1}P^\bot dc,c),\qquad g_1(s^{-1}a,b)=Pb,\qquad K(s^{-1}a,b)=(s^{-1}P^\bot b,0 ).
\end{gathered}\end{equation}

Via homotopy transfer from the Fiorenza-Manetti mapping cocone or the associative mapping cocone there are induced $A_\infty[1]$ algebra structures on $C$ giving us two other models for the homotopy fiber of~$i$. In the particular case of interest to us we may choose a $C$ as above such that moreover $C$ is a square zero graded subalgebra of $B$ (this is a DG associative analog of the setup to Voronov constructions of higher derived brackets \cite{voronov,voronov2}). In the following proposition we make explicit the homotopy transfer formulas under this additional hypothesis, first we introduce a notation: for an integer $k\geq1$ and an ordered partition $k=i_1+\cdots+i_j$ with $j,i_1,\ldots,i_j\geq1$, we define the degree zero map $g^{i_1,\ldots,i_j}\colon\cocone(i)^{\otimes k}\to C$ by
\begin{multline}\label{equ.gconindici} 
g^{i_1,\ldots,i_j}((s^{-1}a_1,b_1)\otimes\cdots\otimes (s^{-1}a_k,b_k))=\\ 
=P(b_1\cdots b_{i_1}\cdot P( b_{i_1+1}\cdots b_{i_1+i_2}\cdot P(\cdots  P(b_{i_1+\cdots+i_{j-1}+1}\cdots b_k)\cdots))),\end{multline}
where the inner $P$s are inserted according to the partition $(i_1,\ldots,i_j)$. For instance:
\[g^1(s^{-1}a,b)=g_1(s^{-1}a,b)=P(b),\qquad g^{2,1}((s^{-1}a_1,b_1)\otimes (s^{-1}a_2,b_2)\otimes (s^{-1}a_3, b_3))=P(b_1\cdot b_2\cdot P(b_3)),\]
\[g^{2,1,2}((s^{-1}a_1,b_1)\otimes\cdots \otimes (s^{-1}a_5,b_5))=P(b_1\cdot b_2\cdot P(b_3\cdot P(b_4\cdot b_5))).\]

\medskip
\begin{theorem}\label{thm.derivedproducts} Given a DG associative algebra $(B,d,\cdot)$ together with a decomposition $B=A\oplus C$ into the direct sum of a DG subalgebra $A$ and a graded subspace $C$ such that $C\cdot C=0$, there is an $A_\infty[1]$ algebra structure $(C,Pd,\Phi(d)_2,0,\ldots,0,\ldots)$ on $C$ given by the derived product
\[\Phi(d)_2(c_1\otimes c_2)=P(dc_1\cdot c_2)\]
together with $A_\infty[1]$ quasi-isomorphisms $F_{As}\colon C\to \cocone(i)_{As}[1]$, 
$G_{As}\colon \cocone(i)_{As}[1]\to C$, $F_\infty\colon C\to \cocone(i)_\infty[1]$, 
$G_\infty\colon \cocone(i)_\infty[1]\to C$
given by $f_{As,1}=f_1=f_{\infty,1}$,
\[f_{As,2}(c_1\otimes c_2)=s^{-1}P^\bot(dc_1\cdot c_2)=f_{\infty,2}(c_1\otimes c_2),\]
$f_{As,k}=0=f_{\infty,k}$ for $k\geq3$. In terms of the linear maps \eqref{equ.gconindici} we have
\[ g_{As,k}=\sum_{i_1+\cdots+ i_j=k}(-1)^{k+j}g^{i_1,\ldots,i_j},\qquad g_{\infty,k}=\sum_{i_1+\cdots+ i_j=k}\frac{(-1)^{k+j}}{i_1!\cdots i_j!}g^{i_1,\ldots,i_j}.\]
The following diagrams are commutative, where $E$ and $L$ are the $A_\infty[1]$ isomorphisms described in  
Lemma~\ref{lem.assvsinftycocone}.
\[\xymatrix{\cocone(i)_\infty\ar@<.4ex>[rr]^-E&&\cocone(i)_{As}\ar@<.4ex>[ll]^-{L}\\ & C\ar[ul]^-{F_\infty}\ar[ur]_-{F_{As}} & }\qquad\xymatrix{\cocone(i)_\infty\ar@<.4ex>[rr]^-E\ar[dr]_-{G_\infty}&&\cocone(i)_{As}\ar@<.4ex>[ll]^-{L}\ar[dl]^-{G_{As}}\\ & C & }\]
\end{theorem}

\begin{proof}
To begin we will show that the given $A_\infty[1]$ structure on $C$ and $A_\infty[1]$ quasi-isomorphisms 
$F_{As}\colon C\to\cocone(i)_{As}[1]$, $G_{As}\colon \cocone(i)_{As}[1]\to  C$ are the ones induced via homotopy transfer from $(\cocone(i)_{As},r_1,r_2,0,\ldots,0,\ldots)$ along the contraction \eqref{contr}. We denote by 
$p_B\colon\cocone(i)[1]\to B$ the projection, then
\[p_Br_2f_1^{\otimes2}(c_1\otimes c_2)=p_Br_2((s^{-1}P^\bot dc_1,c_1)\otimes (s^{-1}P^\bot dc_2,c_2))=(-1)^{|c_1|+1}c_1\cdot P^\bot dc_2=dc_1\cdot c_2,\]
since we have $c\cdot P^\bot b=c\cdot b$ for all $c\in C$, $b\in B$, and $0=d(c_1\cdot c_2)=dc_1\cdot c_2+(-1)^{|c_1|}c_1\cdot dc_2$ for all $c_1,c_2\in C$. For the moment we denote by $(C,Pd,\phi_2,\ldots,\phi_k,\ldots)$ and $F=(f_1,\ldots,f_k,\ldots)$ the $A_\infty[1]$ structure on $C$ and $A_\infty[1]$ morphism $F\colon C\to\cocone(i)_{As}[1]$ induced via homotopy transfer from $\cocone(i)_{As}[1]$, by the above $\phi_2(c_1\otimes c_2)=g_1r_2f_1^{\otimes 2}(c_1\otimes c_2)=P(dc_1\cdot c_2) =\Phi(d)_2(c_1\otimes c_2)$ and $f_2(c_1\otimes c_2)=Kr_2f_1^{\otimes 2}(c_1\otimes c_2)=s^{-1}P^\bot(dc_1\cdot c_2) =f_{As,2}(c_1\otimes c_2)$. Next we see that
\[\begin{split} 
p_Br_2F^2_3(c_1\otimes c_2\otimes c_3)&=p_Br_2(s^{-1}P^\bot(dc_1\cdot c_2)\otimes c_3 + c_1\otimes s^{-1}P^\bot(dc_2\cdot c_3))\\
&=(-1)^{|c_1|+1}c_1\cdot dc_2\cdot c_3 = dc_1\cdot c_2\cdot c_3=0,\end{split}\]
hence $\phi_3=p_1r_2F^2_3=0=Kr_2F^2_3=f_3$. Assuming inductively that $\phi_j=0=f_j$ for all $3\leq j < k$ it is straightforward to see that $p_Br_2F^2_k=0$ and thus also $\phi_k=0=f_k$. It remains to show that the $A_\infty[1]$ morphism $G=(g_1,\ldots,g_k,\ldots)\colon\cocone(i)_{As}[1]\to C$ given by the homotopy transfer formulas is the same as $G_{As}$ in the claim of the proposition. Recall that $g_k$ is defined by the recursion
\begin{equation}\label{recursion} g_k((s^{-1}a_1,b_1)\otimes\cdots\otimes (s^{-1}a_k,b_k))=g_{k-1}R^{k-1}_kK_k\left((s^{-1}a_1,b_1)\otimes\cdots\otimes (s^{-1}a_k,b_k) \right), \end{equation}
where $K_k:=\sum_{j=1}^k \id_{\cocone(i)[1]}^{\otimes j-1}\otimes K\otimes (f_1g_1)^{\otimes k-j}\colon \cocone(i)[1]^{\otimes k}\to\cocone(i)[1]^{\otimes k}$.

Starting  with $g_{As,1}=g_1$ we assume inductively we have shown $g_{As,j}=g_j$ for $1\leq j<k$. We claim that $g_{k-1}((s^{-1}a_1,b_1)\otimes\cdots\otimes(s^{-1}a_{k-1},b_{k-1}))=0$ whenever $b_{k-1}\in C$: this is because there is a bijective correspondence between the set of ordered partitions $k=i_1+\cdots +i_p$ of $k$ with $i_p=1$ and those with $i_p>1$ given by $(i_1,\ldots,i_j,1)\mapsto (i_1,\ldots,i_j+1)$. Moreover when $b_{k-1}\in C$ we have $g^{i_1,\cdots,i_j,1}((s^{-1}a_1,b_1)\otimes\cdots\otimes(s^{-1}a_{k-1},b_{k-1}))=g^{i_1,\cdots,i_j+1}((s^{-1}a_1,b_1)\otimes\cdots\otimes(s^{-1}a_{k-1},b_{k-1}))$ and the two appear with opposite signs in the summation for $g_{k-1}=g_{As,k-1}$. Clearly, we also have $g_{k-1}((s^{-1}a_1,b_1)\otimes\cdots\otimes(s^{-1}a_{k-1},b_{k-1}))=0$ whenever $b_j=0$ for some $1\leq j\leq k-1$, putting these two fact together we see that \eqref{recursion} reduces to:
\[\begin{split} 
g_k((s^{-1}&a_1,b_1)\otimes\cdots\otimes (s^{-1}a_k,b_k))=\\
&=(-1)^{|b_{k-1}|}g_{k-1}((s^{-1}a_1,b_1)\otimes\cdots\otimes r_2((s^{-1}a_{k-2},b_{k-1})\otimes(s^{-1}P^\bot b_{k},0)) )\\ 
&=-g_{k-1}((s^{-1}a_1,b_1)\otimes\cdots\otimes (s^{-1}a_{k-2},b_{k-2})\otimes(\cdots,b_{k-1}\cdot P^{\bot}b_{k}) )\\ 
&=\sum_{i_1+\cdots+i_j=k-1}(-1)^{k+j}g^{i_1,\ldots,i_j}((s^{-1}a_1,b_1)\otimes\cdots\otimes (\cdots, b_{k-1}\cdot b_k-b_{k-1}\cdot Pb_k))\\ 
&=\sum_{i_1+\cdots+i_j=k-1}\!\!\!\!\left( (-1)^{k+j}g^{i_1,\cdots,i_j+1}+(-1)^{k+j+1}g^{i_1,\cdots,i_j,1}\right)((s^{-1}a_1,b_1)\otimes\cdots\otimes (s^{-1}a_k,b_k))\\ 
&=\sum_{i_1+\cdots+i_j=k}\!\!(-1)^{k+j}g^{i_1,\ldots,i_j}((s^{-1}a_1,b_1)\otimes\cdots\otimes (s^{-1}a_k,b_k))\\ 
&=g_{As,k}((s^{-1}a_1,b_1)\otimes\cdots\otimes (s^{-1}a_k,b_k)).
\end{split}\]
This concludes the first part of the proof.

At this point we could similarly prove that $(C,Pd,\Phi(d)_2,0,\ldots,0,\ldots)$, $F_{\infty}$ and $G_\infty$ are the $A_\infty[1]$ structure and $A_\infty[1]$ quasi-isomorphisms induced via homotopy transfer from $\cocone(i)_\infty[1]$ along the contraction \eqref{contr}: this is true but the computation is a little harder and we won't actually need this fact, as we only need to be able to compare the various models. Instead, we observe that the claim of the theorem follows once we show $F_\infty=LF_{As}$ and $G_\infty = G_{As}E$. The first identity is easy and left to the reader; for the second we have
\[\begin{split}
&\sum_{i=1}^k g_{As,i}E^i_k((s^{-1}a_1,b_1)\otimes\cdots\otimes (s^{-1}a_k,b_k))=\\
&=\sum_{j_1+\cdots+ j_i=k}\sum_{i_1+\cdots+i_j=i}(-1)^{i+j}g^{i_1,\ldots,i_j}\left(\left(\cdots,\frac{1}{j_1!}b_1\cdots b_{j_1}\right)\otimes\cdots\otimes\left(\cdots,\frac{1}{j_i!}b_{k-j_i+1}\cdots b_{k}\right)\right)\\
&=\sum_{j_1+\cdots+ j_i=k}\sum_{i_1+\cdots+i_j=i}\frac{(-1)^{i+j}}{j_1!\cdots j_i!}g^{j_1+\cdots+j_{i_1},\ldots,j_{i-i_j+1}+\cdots+j_{i_j}}((s^{-1}a_1,b_1)\otimes\cdots\otimes (s^{-1}a_k,b_k))\\
&=\!\!\!\sum_{i_1+\cdots+i_j=k}\left(\sum_{\tiny\begin{array}{c}h^1_1+\cdots+h^1_{p_1}=i_1\\
\cdots\\
h^j_1+\cdots+h^j_{p_j}=i_j\end{array}}\frac{(-1)^{j+p_1+\cdots + p_j}}{h^1_{1}!\cdots h^1_{p_1}!\cdots h^j_{1}!\cdots h^j_{p_j}!}\right)g^{i_1,\ldots,i_j}((s^{-1}a_1,b_1)\otimes\cdots\otimes (s^{-1}a_k,b_k))
\end{split}\]
and by the equality
\begin{equation}\label{eq:sumper}\begin{split} 
\sum_{i\geq1}\left(\sum_{h_1+\cdots+h_p=i}\frac{(-1)^{p+i}}{h_1!\cdots h_p!}\right)t^i&=
\sum_{p\geq1}\left(\sum_{h\geq1}\frac{(-1)^{h+1}}{h!}t^h\right)^p\\
&=\sum_{p\geq1}(1-e^{-t})^p=\frac{1-e^{-t}}{e^{-t}}=e^t-1=\sum_{i\geq1}\frac{1}{i!}t^i\end{split}
\end{equation}
we get 
\[ \sum_{\tiny\begin{array}{c}h^1_1+\cdots+h^1_{p_1}=i_1\\
\cdots\\
h^j_1+\cdots+h^j_{p_j}=i_j\end{array}}\frac{(-1)^{j+p_1+\cdots + p_j}}{h^1_{1}!\cdots h^1_{p_1}!\cdots h^j_{1}!\cdots h^j_{p_j}!}=\frac{(-1)^{j+i_1+\cdots + i_j}}{i_1!\cdots i_j!}=\frac{(-1)^{j+k}}{i_1!\cdots i_j!}\,.\]
\end{proof}

\begin{remark} In general homotopy transfer from Fiorenza-Manetti mapping cocone or the associative one will induce different $A_\infty[1]$ structures on $C$: the fact that they are the same in the case of the previous proposition follows from the hypothesis $C\cdot C=0$. Moreover, we remark that we only used this hypothesis in the computation of the $A_\infty[1]$ structure on $C$ and the $A_\infty[1]$ morphisms $F_{As},F_\infty$, but never in the computation of $G_{As},G_\infty$: in particular, the formulas for $G_{As},G_\infty$ continue to hold in the more general situation when $C$ is just an algebraic complement of $A$ in $B$.
\end{remark}

\bigskip
\section{Algebraic models of formal period maps}
\label{sec:periodmap}

We recall, from \cite{FMperiods}, the definition and the basic properties of Cartan homotopies.

\begin{definition} A Cartan homotopy between two DG Lie algebras $L$ and $M$ is a linear map 
$\bi\colon L\to M$, $x\mapsto \bi_x$, of degree $-1$  
such that the following formal Cartan identities are satisfied:
\[[\bi_x,\bi_y]=0,\qquad[\bi_x,\bl_y]=\bi_{[x,y]},\qquad\forall x,y\in L, \]
where $\bl\colon L\rh M$, $x\mapsto\bl_x$, is the degree zero map, called the boundary of $\bi$, 
defined by $\bl_x:=d_M\bi_x + \bi_{d_Lx}$. 
Then it is easy to prove that $\bl$ is a morphism of graded Lie algebras: moreover, it is proven in 
\cite[Cor. 3.7]{FMperiods} that when $\bl$ factors through the inclusion $i\colon N\to M$ of a DG-Lie subalgebra then 
$L\to\cocone(i)$, $x\mapsto(\bl_x,s\bi_x)$, is a strict morphism of $L_\infty$ algebras.\end{definition}

\begin{remark} Given a DG Lie algebra $L$, we consider the graded space $\operatorname{Cone}(L):=L\oplus L[1]$. Given an element $x\in L$, we denote by $\bi_x$ its copy in $L[1]\subset\operatorname{Cone}(L)$ and by $\bl_x$ its copy in $L\subset\operatorname{Cone}(L)$. It is straightforward to verify that the differential $d(\bl_x) = \bl_{d_Lx}$, $d(\bi_x)=\bl_x -\bi_{d_Lx}$ and the bracket induced by Cartan's formulas $[\bi_x,\bi_y]=0, [\bi_x,\bl_y]=\bi_{[x,y]}, [\bl_x,\bl_y]=\bl_{[x,y]},\forall x,y\in L$, make $\operatorname{Cone}(L)$ into a DG Lie algebra (cf. with the construction from \cite[p. 291]{Qui}), and that $\bi^{uv}\colon L \to \operatorname{Cone}(L)\colon x\to\bi_x$ is a Cartan homotopy. This is the universal Cartan homotopy, that is, any other Cartan homotopy $\bi: L\to M$ is the composition $\bi = h\circ\bi^{uv}$ of the universal one and a morphism of DG Lie algebras $h:\operatorname{Cone}(L)\to M$. It is also clear that the boundary $\bl: L\to M$ factors as $\bl = \bl^{uv}\circ h$, where $\bl^{uv}\colon x\to\operatorname{Cone}(L)\colon x\to \bl_x$. In particular, the boundary $\bl: L\to M$ of a Cartan homotopy $\bi: L\to M$ is always null-homotopic (by the above factorization, and since $\operatorname{Cone}(L)$ is acyclic).\end{remark}

\begin{example}\label{ex:3.2} Let $X$ be a complex manifold of dimension $n$, $L=KS_X$ the Kodaira-Spencer DG Lie algebra and $M=\End(A^{\ast,\ast}_X)$ the DG Lie algebra of endomorphisms 
of the de Rham complex, then the holomorphic Cartan formulas \eqref{equ.holomorphiccartanformulas} show that $\bi\colon KS_X\to\End(A^{\ast,\ast}_X)$, $\xi\mapsto\bi_\xi$, 
is a Cartan homotopy in the above sense. Moreover, $\bl_\xi$ is the holomorphic Lie derivative with respect to $\xi$ and thus it preserves the Hodge filtration on $A_X^{\ast,\ast}$; in particular, for all $0\leq i\leq n$ the boundary of $\bi$ factors through the inclusion $\chi_i\colon\End(A^{\ast,\ast}_X,A^{\geq i,\ast}_X)\to\End(A^{\ast,\ast}_X)$. The induced strict $L_\infty$ morphism 
$\mathcal{P}^i\colon KS_X\rh\cocone(\chi_i)$ is the algebraic model of the $i$th period map from Theorem~\ref{th:FMmodelperiod}.
\end{example}

The following definition, cf.  \cite{FMAJ}, generalizes  the previous example.
\begin{definition} A \emph{formal period data} is the data of a DG space $(V,d)$, a DG subspace $W\subset V$, a DG-Lie algebra $L$ and a Cartan homotopy $\bi:L\rh\End(V)$ such that the boundary $\bl$ factors through the inclusion $i\colon \End(V,W)\to\End(V)$, where $\End(V,W):=\{f\in\End(V)\,\mbox{ s.t. }\, f(W)\subset W\,\}$. We call the strict $L_\infty$ morphism $L\to\cocone(i)$, $x\mapsto (\bl_x,s\bi_x)$, the formal period map associated to the formal period data. A \emph{split formal period data} is a formal period data together with the choice of a graded subspace $A\subset V$ such that $V=W\oplus A$ as graded spaces.
\end{definition}

Given a split formal period data as above, we denote by $P\colon V\to A$ the projection with kernel $W$ and by $P^\bot=\id_V-P$:  then $\widetilde{P}\colon\End(V)\to\End(V)\colon f\to PfP^\bot$ is a projection with kernel $\End(V,W)$ and image which we may (and will) identify with $\Hom^*(W,A)$. The decomposition $\End(V)=\End(V,W)\oplus\Hom^*(W,A)$ satisfies the hypotheses of Theorem~\ref{thm.derivedproducts}, thus there is a DG associative algebra structure on $\Hom^*(W,A)[-1]$ with differential $\delta(sf)=-s(P[d,f]P^\bot)$ and product $sf\cup sg =s(fdg)$: the associated DG Lie algebra is a model for the homotopy fiber of the inclusion $i$. The symmetrization of the morphism $G_\infty$ of Theorem~\ref{thm.derivedproducts} gives 
an $L_\infty[1]$ quasi-isomorphism from the Fiorenza-Manetti mapping cocone $\cocone(i)[1]$ to $\Hom^*(W,A)$, and the composition of this morphism and the formal period map is the following $L_\infty[1]$ morphism:
\begin{equation}\label{eq:splitperiodmap}\Pi_\infty=(\pi_{1},\ldots,\pi_{k},\ldots)\colon L[1]\rh \Hom^*(W,A),
\end{equation}
\begin{multline}\label{equ.splitperiodmap} \pi_{k}(s^{-1}x_1\odot\cdots\odot s^{-1}x_k)=\\ = \sum_{\sigma\in S_k}\varepsilon(\sigma)\sum_{i_1+\cdots+i_j=k}\frac{(-1)^{k+j}}{i_1!\cdots i_j!} P\bi_{x_{\sigma(1)}}\cdots\bi_{x_{\sigma(i_1)}}P\cdots P\bi_{x_{\sigma(k-i_j+1)}}\cdots\bi_{x_{\sigma(k)}}P^\bot
\end{multline}
We shall call the map \eqref{equ.splitperiodmap} the \emph{split formal period map} associated to the split formal period data.

\begin{example}\label{ex:frompermap} The formal period data in Example~\ref{ex:3.2} splits canonically $A_X^{\ast,\ast}=A_X^{\geq i,\ast}\oplus A_X^{<i,\ast}$. We want to compute the associated split formal period map 
$\Pi_\infty\colon KS_X[1]\rh\Hom^*( A_X^{\geq i,\ast}, A_X^{<i,\ast})$. 
To apply the previous formula for  $\pi_{k}(s^{-1}\xi_1\odot\cdots\odot s^{-1}\xi_k)$ we consider separately the various components 
$\pi_{k}(s^{-1}\xi_1\odot\cdots\odot s^{-1}\xi_k)\colon A_X^{i+j,\ast}\rh A_X^{i+j-k,\ast}$,  where $0\leq j <\min(k,n-i+1)$ (obviously $\pi_{k}(s^{-1}\xi_1\odot\cdots\odot s^{-1}\xi_k)$ vanishes when $j$ is not in this range), 
\begin{multline*} \pi_{k}(s^{-1}\xi_1\odot\cdots\odot s^{-1}\xi_k)=\sum_{\sigma\in S_k}\varepsilon(\sigma)\left( \sum_{i_1+\cdots+i_h=k,\,i_h>j}\frac{(-1)^{h+k}}{i_1!\cdots i_h!}\right)\bi_{\xi_{\sigma(1)}}\cdots\bi_{\xi_{\sigma(k)}}=\\=k!\left(\frac{(-1)^{k+1}}{k!}+\sum_{i_h=j+1}^{k-1}\frac{(-1)^{i_h+1}}{i_h!}\sum_{i_1+\cdots+i_{h-1}=k-i_h}\frac{(-1)^{(h-1)+(k-i_h)}}{i_1!\cdots i_{h-1}!}\right)\bi_{\xi_1}\cdots\bi_{\xi_k}. 
\end{multline*}
and then, according to \eqref{eq:sumper}, 
\[\begin{split}
\pi_{k}(s^{-1}\xi_1\odot\cdots\odot s^{-1}\xi_k)&=
\left(-k!\sum_{i_h=j+1}^k \frac{(-1)^{i_h}}{i_h!(k-i_h)!}\right)\bi_{\xi_1}\cdots\bi_{\xi_k}\\ 
&=\left(\sum_{h=0}^{j}(-1)^{h}\binom{k}{h}\right) \bi_{\xi_1}\cdots\bi_{\xi_k}.
\end{split}\]
\end{example}

For  $i=n$,  the computation of Example~\ref{ex:frompermap}  gives:

\begin{theorem}\label{th:modelperiod2} Let $X$ be a complex manifold of dimension $n$. Then the 
$L_\infty[1]$ morphism 
\[\Pi_\infty=(\pi_1,\ldots,\pi_k,\ldots)\colon KS_X[1]\rh\Hom^*(A_X^{n,\ast},A_X^{<n,\ast}),\] 
with  Taylor coefficients  
\[\pi_{k}(s^{-1}\xi_1\odot\cdots\odot s^{-1}\xi_k)=\bi_{\xi_1}\cdots\bi_{\xi_k}\colon A_X^{n,\ast}\to A_X^{n-k,\ast},\qquad 0<k\le n,\] 
is an algebraic model of the $n$th period map.
\end{theorem}

Suppose now that $X$ is a compact K\"ahler manifold; as above  we denote by 
$H_X^{\ast,\ast}$ the space of harmonic forms, by 
$\imath\colon H_X^{\ast,\ast}\rh A_X^{\ast,\ast}$ and $\pi\colon A_X^{\ast,\ast}\rh H_X^{\ast,\ast}$ the inclusion and the harmonic projection and by $h=-\debar^\ast G_\debar$ the $\debar$-propagator 
(Definition~\ref{def:propagator}). 
As usual we denote by $P\colon A_X^{\ast,\ast}\rh A_X^{<p,\ast}$ the projection with kernel $A_X^{\geq p,\ast}$. 
We have the following contraction
\begin{equation}\label{contr.per.map}
\xymatrix{\left(\Hom^*(H_X^{\geq p,\ast},H_X^{<p,\ast}), 0\right)\ar@<.4ex>[r]^-{i}&\left(\Hom^*(A^{\geq p,\ast}_X,A_X^{<p,\ast}),\delta\right)\ar@<.4ex>[l]^-{g_1}\ar@(ul,ur)[]^K},\qquad \delta(f)=Pdf-(-1)^{|f|}fd,
\end{equation}
\[  i(f)=\imath f\pi,\qquad g_1(f)=\pi f\imath,\qquad K(f)=hf+
(-1)^{|f|} \imath\pi f h\,.\]
In fact 
\[\begin{split}
(\delta K+K\delta)(f)&= \delta(hf+(-1)^{|f|}\imath\pi fh) + K(Pdf-(-1)^{|f|}fd)\\
&=Pdhf+(-1)^{|f|}(hf+(-1)^{|f|}\imath\pi fh)d+hPdf-(-1)^{|f|}hfd+\imath\pi fdh\\
&=Pdhf+hPdf+\imath\pi f(dh+hd)\,.
\end{split}
\]
Next we notice that $dh+hd=\imath\pi-\id$ and $[P,h]=0$, thus $Phdf + hPdf = P(\imath\pi-\id)f=(\imath\pi-\id)f$ and
\[(\delta K+K\delta)(f)=(\imath\pi-\id)f + \imath\pi f(\imath\pi-\id)=\imath\pi f\imath\pi - f = ig_1(f)- f.\]
The side conditions $g_1K=K^2=Ki=0$ are also easily verified.

\begin{proposition}\label{prop.transfertoharmonic}
The homotopy transfer along the contraction \eqref{contr.per.map} induces: the trivial $L_\infty[1]$ structure on $\Hom^*(H_X^{\geq p,\ast},H_X^{<p,\ast})$, the $L_\infty[1]$ quasi-isomorphism 
\[G=(g_1,\ldots,g_n,\ldots)\colon \Hom^*(A_X^{\geq p,\ast},A_X^{<p,\ast})\to \Hom^*(H_X^{\geq p,\ast},H_X^{<p,\ast}),\]
\[g_k(f_1\odot\cdots\odot f_k)=\sum_{\sigma\in S_k}\varepsilon(\sigma)\pi f_{\sigma(1)}h\de f_{\sigma(2)}h\de\cdots h\de f_{\sigma(k)}\imath,\]
and the strict $L_\infty[1]$ quasi-isomorphism $i\colon \Hom^*(H_X^{\geq p,\ast},H_X^{<p,\ast})\to \Hom^*(A_X^{\geq p,\ast},A_X^{<p,\ast})$.
\end{proposition}
\begin{remark} One should notice that in the above formula the $f_j\in\Hom^*(A_X^{\geq p,\ast},A_X^{<p,\ast})$ are identified with endomorphisms $f_j\in\End(A_X^{\ast,\ast})$ such that $\operatorname{Im}(f_j)\subset A_X^{<p,\ast}\subset \operatorname{Ker}(f_j)$. The same remark applies to the following computations.\end{remark}

\begin{proof} In the computation of the homotopy transfer we take advantage of the fact that the $L_\infty[1]$ structure on $\Hom^*(A_X^{\geq p,\ast},A_X^{<p,\ast})$ is the symmetrized of the $A_\infty[1]$ structure $(q_1,q_2,0,\ldots,0,\ldots)$, \[q_2(f_1\otimes f_2)=(-1)^{|f_1|+1}f_1df_2=(-1)^{|f_1|+1}f_1\de f_2,\]
where the last equality follows by the previous remark. For the moment we denote the $A_\infty[1]$ structure induced on $\Hom^*(H_X^{\geq p,\ast},H_X^{<p,\ast})$ via homotopy transfer along \eqref{contr.per.map} by $(0,r_2,\ldots,r_k,\ldots)$, and we denote by 
$(i,i_2,\ldots,i_k,\ldots)\colon\Hom^*(H_X^{\geq p,\ast},H_X^{<p,\ast})\rh \Hom^*(A_X^{\geq p,\ast},A_X^{<p,\ast})$ 
the induced $A_\infty[1]$ quasi-isomorphism. Since $d\imath=\pi d=0$ we see that $q_2i^{\otimes 2}=0$, and thus $r_2=g_1q_2i^{\otimes2}=0=Kq_2i^{\otimes 2}=i_2$. Since $q_k=0$ for $k\geq3$ a straightforward induction shows that $i_k=r_k=0$ for all $k\geq2$. 
It remains to prove that the induced $A_\infty[1]$ quasi-isomorphism 
\[G=(g_1,\ldots,g_k,\ldots)\colon\Hom^*(A_X^{\geq p,\ast},A_X^{<p,\ast})\to\Hom^*(H_X^{\geq p,\ast},H_X^{<p,\ast})\] 
is given by
\[ g_k(f_1\otimes\cdots\otimes f_k)=\pi f_1h\de f_2h\de\cdots h\de f_k\imath,\]
that is, that the $g_k$ satisfy the recursive relation
\begin{equation}\label{eq:trans.sec.per}
g_k(f_1\otimes\cdots\otimes f_k)=g_{k-1}Q^{k-1}_kK_k\left(f_1\otimes\cdots\otimes f_k \right),
\end{equation}
where as usual $K_k=\sum_{j=0}^{k-1}\id^{\otimes j}\otimes K\otimes(ig_1)^{\otimes k-j-1}$. We first consider the case $k=2$: as  for all $f_1,f_2\in\Hom^*(A_X^{\geq p,\ast},A_X^{<p,\ast})$ we have $q_2(f_1\otimes\imath\pi f_2)=0$ it follows that $Q^1_2K_2(f_1\otimes f_2)=(-1)^{|f_1|}q_2(f_1\otimes K(f_2))$ and thus
\[g_2(f_1\otimes f_2)= (-1)^{|f_1|}g_1q_2(f_1\otimes (hf_2+(-1)^{|f_2|}\imath\pi f_2h))=-g_1(f_1\de h f_2)=\pi f_1h\de f_2\imath\,.\]
We suppose inductively that $g_1,\ldots, g_{k-1}$ are of the desired form: in particular for all $f\in\Hom^*(A_X^{\geq p,\ast},A_X^{<p,\ast})$ we have $g_{k-1}(\cdots\otimes \imath\pi f )=g_{k-1}(\cdots\otimes K(f))=0$ and thus the right hand side of \eqref{eq:trans.sec.per} becomes
\[\begin{split}
g_{k-1}Q^{k-1}_kK_k\left(f_1\otimes\cdots\otimes f_k \right)&=(-1)^{|f_{k-1}|}g_{k-1}(f_1\otimes\cdots\otimes f_{k-2}\otimes q_2(f_{k-1}\otimes K(f_k)))\\ 
&=g_{k-1}(f_1\otimes\cdots\otimes f_{k-2}\otimes(f_{k-1}h\de f_k))=\pi f_1h\de f_2h\de\cdots h\de f_k\imath.
\end{split}\]
\end{proof}

We are now able to present a third algebraic model of the $n$th period map with the nice property that the target $L_\infty$ algebra is minimal (in fact, it has the trivial $L_\infty$ algebra structure).

\begin{theorem}\label{th:modelperiodmap} The $L_\infty[1]$ morphism 
$\mathcal{P}_\infty=(p_1,\ldots,p_k,\ldots)\colon KS_X[1]\to\Hom^*(H_X^{ n,\ast},H_X^{<n,\ast})$, 
from the Kodaira-Spencer DG-Lie algebra to $\Hom^*(H_X^{n,\ast},H_X^{<n,\ast})$ with the trivial $L_\infty[1]$ structure, given in Taylor coefficients by
\begin{multline}\label{equ.minimalnperiodmap} 
p_{k}(s^{-1}\xi_1\odot\cdots\odot s^{-1}\xi_k)=\\=\sum_{j=1}^k\sum_{\sigma\in S(j,1,\ldots,1)}\varepsilon(\sigma)\pi\bi_{\xi_{\sigma(1)}}\bi_{\xi_{\sigma(2)}}\cdots\bi_{\xi_{\sigma(j)}}h\bl_{\xi_{\sigma(j+1)}}h\bl_{\xi_{\sigma(j+2)}}\cdots h\bl_{\xi_{\sigma(k)}}\imath,
\end{multline}
is an algebraic model of the $n$th period map.
\end{theorem}
\begin{proof} We have to show that the map $\mathcal{P}_\infty$ defined in 
\eqref{equ.minimalnperiodmap} is the composition of the model of the $n$th period map 
$\Pi_\infty\colon KS_X[1]\rh\Hom^*(A_X^{n,\ast},A_X^{<n,\ast})$, 
$\pi_k(s^{-1}\xi_1\odot\cdots\odot s^{-1}\xi_k)=\bi_{\xi_1}\cdots\bi_{\xi_k}$, 
of Theorem~\ref{th:modelperiod2} and the $L_\infty[1]$ quasi-isomorphism 
$G\colon\Hom^*(A_X^{n,\ast},A_X^{<n,\ast})\rh \Hom^*(H_X^{n,\ast},H_X^{<n,\ast})$ 
of Proposition~\ref{prop.transfertoharmonic}. 

First of all we notice that $g_k(f_1\odot\cdots\odot f_k)=0$ whenever at least  two entries $f_i$, $f_j$ belong to the subspace $\Hom^*(A_X^{n,\ast},A_X^{<n-1,\ast})\subset\Hom^*(A_X^{n,\ast},A_X^{<n,\ast})$. Thus the composition $G\Pi_\infty$ is given by
\[\begin{split}\sum_{i=1}^k\frac{1}{i!}&\sum_{j_1+\cdots+j_i=k}\sum_{\sigma\in S(j_1,\ldots,j_i)}\varepsilon(\sigma)g_i(\bi_{\xi_{\sigma(1)}}\cdots\bi_{\xi_{\sigma(j_1)}}\odot\cdots\odot\bi_{\xi_{\sigma(k-j_i+1)}}\cdots\bi_{\xi_{\sigma(k)}})=\\
&=\sum_{j=1}^{k}\sum_{\sigma\in S(j,1,\ldots,1)}\varepsilon(\sigma)\frac{1}{(k-j)!}g_{k-j+1}(\bi_{\xi_{\sigma(1)}}\cdots\bi_{\xi_{\sigma(j)}}\odot \bi_{\xi_{\sigma(j+1)}}\odot\cdots\odot \bi_{\xi_{\sigma(k)}})\\
&=\sum_{j=1}^k\sum_{\sigma\in S(j,1,\ldots,1)}\varepsilon(\sigma)\pi\bi_{\xi_{\sigma(1)}}\cdots\bi_{\xi_{\sigma(j)}}h\de\bi_{\xi_{\sigma(j+1)}}h\de \cdots h\de\bi_{\xi_{\sigma(k)}}\imath\\ 
&=\sum_{j=1}^k\sum_{\sigma\in S(j,1,\ldots,1)}\varepsilon(\sigma)\pi\bi_{\xi_{\sigma(1)}}\cdots\bi_{\xi_{\sigma(j)}}h\bl_{\xi_{\sigma(j+1)}}h\cdots h\bl_{\xi_{\sigma(k)}}\imath=p_k(s^{-1}\xi_1\odot\cdots\odot s^{-1}\xi_k)\,.
\end{split}\]
To justify the last passage we claim that $h\de\bi_{\xi_1}h\de\cdots h\de\bi_{\xi_k}\imath=h\bl_{\xi_1}h\cdots h\bl_{\xi_k}\imath$ for all $k\geq1$ and $\xi_1,\ldots,\xi_k\in KS_X$. In fact, since $\de\imath=\de h\de=0$, we have
\[\begin{split} 
h\bl_{\xi_1}h\cdots h\bl_{\xi_k}\imath &= h\bl_{\xi_1}h\cdots h\bl_{\xi_{k-1}}h(\de\bi_{\xi_k}\mp\bi_{\xi_k}\de)\imath\\ 
&= h\bl_{\xi_1}h\cdots h(\de\bi_{\xi_{k-1}}\mp\bi_{\xi_{k-1}}\de)h\de\bi_{\xi_k}\imath = \cdots \\
&= h(\de\bi_{\xi_1}\mp\bi_{\xi_1}\de)h\de\bi_{\xi_2}h\de\cdots h\de\bi_{\xi_{k-1}}h\de\bi_{\xi_k}\imath =  h\de\bi_{\xi_1}h\de\cdots h\de\bi_{\xi_k}\imath\,.\end{split}\]
\end{proof}

\begin{remark} Taking the composition of the morphism $\sP_{\infty}$ of Theorem~\ref{th:modelperiodmap}
with the natural projection $\Hom^*(H_X^{ n,\ast},H_X^{<n,\ast})\to \Hom^*(H_X^{ n,\ast},H_X^{n-1,\ast})$ 
we recover the $L_{\infty}[1]$ morphism introduced in \cite{CCK} for proving that 
the obstructions to deformations of $X$ are annihilated by the map
\[ \bi\colon H^2(X,\Theta_X)\to \prod_i\Hom(H^i(X,\Omega_X^n),H^{i+2}(X,\Omega_X^{n-1}))\,.\] 

Taking the composition of the morphism $\sP_{\infty}$ of Theorem~\ref{th:modelperiodmap}
with the natural projection $\Hom^*(H_X^{ n,\ast},H_X^{<n,\ast})\to \Hom^*(H_X^{n,\ast},H_X^{0,\ast})$ 
we get the  $L_\infty[1]$ morphism 
\[\Phi=(\phi_1,\ldots,\phi_k,\ldots)\colon KS_X[1]\to\Hom^*(H_X^{n,\ast},H_X^{0,\ast}),\] 
where $\phi_k=0$ for $k<n$ and 
\[\phi_{k}(s^{-1}\xi_1\odot\cdots\odot s^{-1}\xi_k)=\sum_{\sigma\in S(n,1,\ldots,1)}\varepsilon(\sigma)\pi\bi_{\xi_{\sigma(1)}}\bi_{\xi_{\sigma(2)}}\cdots\bi_{\xi_{\sigma(n)}}h\bl_{\xi_{\sigma(j+1)}}h\bl_{\xi_{\sigma(j+2)}}\cdots h\bl_{\xi_{\sigma(k)}}\imath
\]
for $k\ge n$, which is the (formal, pointed) analog of the Yukawa potential function 
(see e.g. \cite{VoisinSNS}) for an arbitrary compact K\"{a}hler manifold.
\end{remark}

\bigskip
\section{Semidirect products of $L_\infty[1]$ algebras  and homotopy pull-backs}
\label{sec.semidirect}

It is known (see \cite{Hin,Pri}) that the category of DG cocommutative conilpotent coalgebras over a field $\K$ (of characteristic zero, as always) carries a model category structure, such that the fibrant objects are precisely the $L_\infty[1]$ algebras. In particular, given a pair of $L_\infty[1]$ morphisms $L\to M$ and $N\to M$ we can form their homotopy fiber product $L\times^h_M N$ via standard model categorical techniques. The aim of this section is to present, under some assumptions, an explicit construction of homotopy fiber products, based on the theory of $L_\infty$ extensions developed in  \cite{ChLaz2,methazambon}, and on Voronov's theory of higher derived brackets \cite{voronov,voronov2,derived,bordemann}.

Let $(L,q_1,\ldots,q_k,\ldots)$ be an $L_{\infty}[1]$ algebra and $I\subset L$ an $L_\infty[1]$ ideal, i.e.,
a graded subspace $I\subset L$ such that  $q_k(I\otimes L^{\odot k-1})\subset I$ for every  $k\geq1$. 
Then there is a unique induced $L_\infty[1]$ algebra structure on the quotient $L/I$ such that the projection 
$L\to L/I$ is a strict morphism of $L_\infty[1]$ algebras.

We extend the usual  construction of semidirect products of Lie algebras 
to $L_{\infty}[1]$ algebras, following \cite{ChLaz2,methazambon}.  
Let $(M,r_1,\ldots,r_n,\ldots)$ and  $(I,q_1,\ldots,q_n,\ldots)$ be two $L_\infty[1]$ algebras: we denote by 
$\CE(I)$ the Chevalley-Eilenberg DG-Lie algebra of $I$, and consider an $L_\infty[1]$ morphism  
\[\phi=(\phi_1,\ldots,\phi_k,\ldots)\colon M\to\CE(I)[1]\]
The semidirect product $I\rtimes_{\phi} M$ is the graded vector space $I\times M$,  
equipped with the $L_\infty[1]$ algebra structure defined in Taylor coefficients $\widetilde{q}_k\colon(I\times M)^{\odot k}\to I\times M$ by:
\[\widetilde{q}_k(i_1\odot\cdots\odot i_k)=(q_k(i_1\odot\cdots\odot i_k),0),\]
\[ \widetilde{q}_j(m_1\odot\cdots\odot m_j)=(s\phi_j(m_1\odot\cdots\odot m_j)_0(1),r_j(m_1\odot\cdots\odot m_j)),\]
\[ \widetilde{q}_{j+k}(m_1\odot\cdots\odot m_j\otimes i_1\odot\cdots\odot i_k)=(s\phi_j(m_1\odot\cdots\odot m_j)_k(i_1\odot\cdots\odot i_k),0),\]
where $s\phi_j$ is the composition $M^{\odot j}\xrightarrow{\phi_j}\CE(I)[1]\xrightarrow{\,s\,}\CE(I)$. 
\begin{theorem}\label{thm.Linfinityactions} 
	In the above setup, the coderivation $\widetilde{Q}=(\widetilde{q}_1,\ldots,\widetilde{q}_n,\ldots)$ defines indeed an $L_\infty[1]$ algebra structure on $I\rtimes_{\phi} M$. Furthermore, the natural inclusion 
	$I\to I\rtimes_{\phi} M$ and projection $I\rtimes_{\phi} M\to M$ are strict $L_\infty[1]$ morphisms. Conversely, every 
	$L_\infty[1]$ algebra structure on $I\times M$ having these properties is the semidirect product 
	$I\rtimes_{\phi}M$ for a unique $L_\infty[1]$ morphism $\phi\colon M\to\CE(I)[1]$.
\end{theorem}

\begin{proof} For a proof we refer either to \cite[Prop. 3.5]{ChLaz2} or to the paper \cite{methazambon}.
\end{proof}

The above result can be used to give an explicit construction of homotopy fiber products in the category of 
$L_\infty[1]$ algebras and $L_\infty[1]$ morphisms: we remark  that this category 
is not complete, the problem being that in general the equalizer of two $L_\infty[1]$ morphisms may not exist. 
We follows essentially the argument used in Schuhmacher's thesis 
\cite{schuhmacher}. Recall that, by definition, a morphism of  $L_\infty[1]$ algebras 
$F=(f_1,f_2,\ldots)\colon L\rh M$ is a fibration  if its linear part $f_1$ is surjective.

\begin{lemma}\label{lem:fibrvsstrictfibr} Every fibration $N\to M$ of $L_\infty[1]$ algebras admits a factorization $N\rh\widetilde{N}\rh M$ where the first arrow is an $L_\infty[1]$ isomorphism and the second one is a strict fibration.\end{lemma}

\begin{proof}
	This is proved in full details in \cite[Lemma 1.5.4]{schuhmacher}: for the reader convenience we give a
	sketch of the proof.   
	Denoting by $V$ and $W$ the underlying complexes of $N$ and $M$ respectively, and by $f_k\colon V^{\odot k}\to W$, $k>0$, the Taylor coefficients of the fibration, 
	since $f_1$ is surjective there exist a sequence of maps 
	$g_k\colon V^{\odot k}\to V$, $k>0$, such that $g_1=Id$ and $f_1g_k=f_k$ for every $k$.
	These maps induce an isomorphism of graded coalgebras $G=(g_1,\ldots)\colon S(V)\to S(V)$ and 
	we define $\widetilde{N}$ as the unique $L_\infty[1]$ structure on $V$ such that $G\colon N\to \widetilde{N}$ is an isomorphism of $L_\infty[1]$ algebras.
\end{proof}

\begin{proposition}\label{prop:pullbackbyfibration} 
	Given the $L_\infty[1]$ algebras $L$, $M$, $N$, an $L_\infty[1]$ morphism $F\colon L\to M$ and a fibration $G\colon N\to M$, the fiber product $L\times_M N$ exists in the category of $L_\infty[1]$ algebras and $L_\infty[1]$ morphisms.
\end{proposition}

\begin{proof} According to Lemma~\ref{lem:fibrvsstrictfibr} it is not restrictive to assume that $G=g\colon 
	N\to M$ is a strict fibration;  thus the kernel $I:=\operatorname{Ker}(g)$ is an $L_\infty[1]$ ideal and therefore by Theorem~\ref{thm.Linfinityactions} we have $N=I\rtimes_{\phi}M$ for a well defined  
	$L_\infty[1]$ morphism $\phi\colon M\rh\CE(I)[1]$. Considering the composite morphism  
	$\psi=\phi\circ F\colon L\to\CE(I)[1]$, we claim that the  diagram 
	\[\xymatrix{ I\rtimes_{\psi} L\ar[r]^-{\widetilde{F}}\ar[d] & I\rtimes_{\phi} M\ar[d] \\ 
		L\ar[r]^-F & M}\]
	is  cartesian  in the category of 
	$L_\infty[1]$ algebras and $L_\infty[1]$ morphisms,
	where the vertical arrows are the projections (which are strict $L_\infty[1]$ morphisms) and $\widetilde{F}$ is given in Taylor coefficients by
	\[ \widetilde{f}_1(i,l)=(i,f_1(l)),\qquad\widetilde{f}_k((i_1,l_1)\odot\cdots\odot(i_k,l_k))=(0,f_k(l_1\odot\cdots\odot l_k))\,\,\,\mbox{for $k\geq2$}\,:\]
	it is straightforward to see that $\widetilde{F}$ is an $L_\infty[1]$ morphism  
	\cite[Prop. 1.2.4]{schuhmacher}. 
	Given a commutative diagram of $L_{\infty}[1]$ algebras
	\[\xymatrix{ X\ar[r]^-{H}\ar[d]_K & I\rtimes_{\phi} M\ar[d] \\ L\ar[r]^-F & M} \]
	it is easy to see that the morphism of graded coalgebras $S(X)\to S(I\rtimes_{\psi} L)$ 
	given in Taylor coefficients by $(p_Ih_j,k_j)\colon X^{\odot j}\to I\rtimes_{\psi} L$, 
	where $p_I\colon I\rtimes_\phi M\to I$ is the projection, is the only one making the required diagram commutative, and again one should check that this is an $L_\infty[1]$ morphism. We leave to the reader to fill in the details of the easy computations.
\end{proof}

\begin{remark} The claim of the previous proposition also follows from the results of \cite{Hin,Pri} and standard facts on
	model categories. On the other hand, we shall need the explicit construction in terms of semidirect products in what follows. We also point out that  Proposition~\ref{prop:pullbackbyfibration} is the only non-trivial step to prove  
that the category of $L_\infty[1]$ algebras and $L_\infty[1]$ morphisms is a pointed category of fibrant objects, as defined in \cite{kenbrown} (of course this already follows from the results in \cite{Hin,Pri}, but going this way one avoids many technical details). In fact, according to Brown's 
factorization lemma \cite[p. 421]{kenbrown}, in order to prove the key property that every $L_\infty[1]$ morphisms $f:L\to M$ can be factored into the composition $L\xrightarrow{\sim} N \twoheadrightarrow M$ of a weak equivalence and a fibration, it is sufficient to do so for the diagonal $M\to M\times M$. This is easy (the reason being that the diagonal is a \emph{strict} $L_\infty[1]$ morphism): for instance, we can take the factorization
\[ M \xrightarrow{\sim} M[t,dt]\twoheadrightarrow M\times M,\]
where the left hand side arrow sees an element $m\in M$ as a constant form in $M[t,dt]$, and the right hand side arrow evaluates a form at $t=0$ and $t=1$. The other axioms for a category of fibrant objects are straightforwardly verified. \end{remark} 

In particular, it makes sense the notion of homotopy fiber product 
$L\times^h_M N$ of a pair of $L_\infty[1]$ morphisms $F\colon L\to M$ and $G\colon N\to M$. 
Notice that it is properly defined only up to quasi-isomorphism and:
\begin{enumerate}
	
	\item if $G$ is a fibration we simply take the fiber product $L\times_M N$ as in 
	Proposition~\ref{prop:pullbackbyfibration}; in particular, if $G$ is a strict fibration then 
	$N=I\rtimes_\phi M$ for some $\phi\colon M\to\CE(I)[1]$ and therefore $I\rtimes_{\phi F}L$ is a model for 
	$L\times^h_M N$;
	
	\item if $G$ is not a fibration, consider any factorization $G\colon N\xrightarrow{i}\widetilde{N}\xrightarrow{\widetilde{G}}M$, with $i$ a quasi-iomorphism and $\widetilde{G}$ a fibration: then $L\times_M \widetilde{N}$
	is a model for $L\times^h_M N$.
\end{enumerate}

Based on the previous considerations, we shall give an explicit construction of a (small) model of the homotopy fiber product when one of the two morphisms is an inclusion of DG Lie algebras 
$i\colon N\to M$.
We fix a graded vector subspace $A\subset M$ such that $M=A\oplus N$, 
and denote by $d$ the differential on $M$ and by $P\colon M\to A$ the projection with kernel $N$. 
Since  the complex $(A,Pd)$ is a deformation retract of the (desuspended) mapping cocone $\cocone(i)[1]$ 
(see \eqref{contr}),  
homotopy transfer induces an $L_\infty[1]$ algebra structure $\phi(d)$ on $A$, which is a model for the homotopy fiber of the inclusion $i$. Moreover, according to \cite[Rem. 5.9]{derived} (cf. also \cite{bordemann}), 
there exists  a DG-Lie algebra morphism
\[\phi\colon (M,d,[-,-])\to(\CE(A),[\phi(d),-],[-,-])\] 
such that the map 
$N\to A[-1]\rtimes_{\phi} M$, $n\mapsto (0,n)$, is a strict quasi-isomorphism of 
$L_\infty$ algebras and thus the projection 
$A[-1]\rtimes_{\phi} M\to M$ is weakly equivalent to the inclusion $i\colon N\to M$. 
Hence, according to Proposition~\ref{prop:pullbackbyfibration}, for any $L_\infty$ morphism $L\to M$ the $L_\infty$ algebra 
$A[-1]\rtimes_{\phi F} L$ is a model for the homotopy pullback $L\times^h_{M}N$. Finally, explicit formulas for the $L_\infty[1]$ structure on $A$ and the morphism $\phi$ were given 
in \cite{derived}  under the additional assumption that $A\subset M$ is a graded Lie subalgebra (see \cite{bordemann} for the general case). In the particular case when $A\subset M$ is an \emph{abelian} graded Lie subalgebra, one recovers the $L_\infty[1]$ algebra structure on $A$ given by Voronov's second construction of higher derived brackets \cite{voronov2}
\[ \phi(d)_1(a)=Pda,\qquad\phi(d)_k(a_1\odot\cdots\odot a_k)=P[\cdots[da_1,a_2]\cdots,a_k]\,\,\mbox{ for }k\geq2,\]
and the morphism $\phi\colon M\to\CE(A)$, $m\mapsto (\phi(m)_0,\ldots,\phi(m)_k,\ldots)$, of DG-Lie algebras given by Voronov's first construction of higher derived brackets \cite{voronov}
\[\phi(m)_0(1)=Pm,\qquad \phi(m)_k(a_1\odot\cdots\odot a_k)=P[\cdots[m,a_1]\cdots,a_k]\,\,\mbox{ for }k\geq1. \]
In particular, if $A\subset M$ is an abelian graded Lie subalgebra, then    
$A[-1]\rtimes_\phi M$ is exactly the the $L_\infty$ algebra defined in \cite[Sec. 4, Thm. 2]{voronov2}. 
Summing together the results of this section we finally reach our goal, expressed by the following theorem:

\begin{theorem}\label{thm.modelfiberproduct} Let $(M,d,[-,-])$ be a DG-Lie algebra, $N\subset M$ a DG-Lie subalgebra  and $A\subset M$ an abelian graded Lie subalgebra such that $M=N\oplus A$ as graded vector spaces: denote by 
	$P\colon M\to A$ the projection with kernel $N$. 
	Given another DG-Lie algebra $L$ and an $L_\infty[1]$ morphism 
	$F=(f_1,\ldots,f_k,\ldots)\colon L[1]\to M[1]$, the $L_\infty[1]$ algebra 
	\[ (A\rtimes_{\phi F} L[1],q_1,\ldots,q_k,\ldots),\] 
	where the brackets are 
	\[\begin{split}q_1(a,s^{-1}x)&=\left( P(da + sf_1(s^{-1}x)),-s^{-1}dx \right),\\ 
	q_2(s^{-1}x_1\odot s^{-1}x_2)&=\left(  Psf_2(s^{-1}x_1\odot s^{-1}x_2 ),(-1)^{|x_1|}s^{-1}[x_1,x_2] \right),\\
	q_j(s^{-1}x_1\odot\cdots\odot s^{-1}x_j)&=\left(Psf_j(s^{-1}x_1\odot\cdots\odot s^{-1}x_j),0\right)\,\,j\geq3,\\
	q_k(a_1\odot\cdots\odot a_k)&=\left(P[\cdots[da_1,a_2]\cdots,a_k],0\right)\,\,k\geq2,\\
	q_{j+k}(s^{-1}x_1\odot\cdots\odot s^{-1}x_j \otimes a_1\odot\cdots\odot a_k) &=\left(P[\cdots[sf_j(s^{-1}x_1\odot\cdots\odot s^{-1}x_j) , a_1]\cdots,a_k],0\right)\,\, j,k\geq1,\end{split} \]
	(here $sf_j$ is the composition $L[1]^{\odot j}\xrightarrow{f_j}M[1]\xrightarrow{s}M$) is a model for the homotopy fiber product $L\times^h_M N$  
	.
\end{theorem}

\bigskip
\section{$L_\infty$ models for Yukawa algebras}

Let $X$ be a K\"ahler manifold of dimension $n\geq2$ with a fixed K\"ahler metric. In  Theorem \ref{th:modelperiodmap} we constructed an $L_\infty[1]$ model  
\[ \xymatrix{ &\Hom^*(H_X^{n,\ast},H_X^{0<\ast<n,\ast})\ar[d] \\ KS_X[1]\ar[r]^-{\mathcal{P}_\infty} & \Hom^*(H_X^{n,\ast},H_X^{<n,\ast})} \]
of the geometric diagram of pointed moduli spaces
\[ \xymatrix{&(\Grass(F^1H^*(X,\C)),F^nH^*(X,\C))\ar[d]\\
B\ar[r]^-{\sP^n}&(\Grass(H^*(X,\C)),F^nH^*(X,\C))}\]
where $B$ is the base of the semiuniversal deformation of $X$ and $\mathcal{P}^n$ the $n$th local period map. 
As explained in the introduction, if we want to replicate the geometric diagram 
\[ \xymatrix{Y_B\ar[r]\ar[d]&(\Grass(F^1H^*(X,\C)),F^nH^*(X,\C))\ar[d]\\
B\ar[r]^-{\sP^n}&(\Grass(H^*(X,\C)),F^nH^*(X,\C))}\]
in the category of $L_\infty$ algebras, the only thing that makes sense (according to the general philosophy of derived deformation theory) is to replace the above cartesian diagram of complex singularities by a homotopy cartesian diagram of $L_\infty[1]$ algebras
\[ \xymatrix{\Yuk_X[1]\ar[r]\ar[d] &\Hom^*(H_X^{n,\ast},H_X^{0<\ast<n,\ast})\ar[d] \\ KS_X[1]\ar[r]^-{\mathcal{P}_\infty} & \Hom^*(H_X^{n,\ast},H_X^{<n,\ast})} \]

\begin{definition} We shall denote a homotopy fiber product of the previous diagram of $L_\infty$ algebras by $\Yuk_X$ (thus, strictly speaking, $\Yuk_X$ denotes an object in the homotopy category of $L_\infty$ algebras). In particular the associated deformation functor $\Def_{\Yuk_X}\colon \Art\to\Set$ is well defined.\end{definition}

By a straightforward application of Theorem~\ref{thm.modelfiberproduct} we obtain the following concrete $L_\infty[1]$ model of $\Yuk_X[1]$.

\begin{theorem}\label{th:L_ooalgyukawa} The $L_\infty[1]$ algebra $(KS_X[1]\times\Hom^*(H_X^{n,\ast},H_X^{0,\ast})[-1],q_1,\ldots,q_k,\ldots)$, where the nontrivial brackets are $q_1(s^{-1}\xi,sf)=(s^{-1}\debar\xi,0)$,
\[ q_2(s^{-1}\xi_1\odot s^{-1}\xi_2)=\left\{\begin{array}{ll}
((-1)^{|\xi_1|}s^{-1}[\xi_1,\xi_2],s\pi\bi_{\xi_1}\bi_{\xi_2}\imath) & \mbox{if $n=2$,}\\
(-1)^{|\xi_1|}s^{-1}[\xi_1,\xi_2] & \mbox{if $n >2$,}
\end{array}\right. \]
$q_k=0$ for $2<k<n$ and finally
\begin{equation*}q_{n+k}(s^{-1}\xi_1\odot\cdots\odot s^{-1}\xi_{n+k})=\sum_{\sigma\in S(n,1,\ldots,1)}\varepsilon(\sigma)s\pi\bi_{\xi_{\sigma(1)}}\cdots\bi_{\xi_{\sigma(n)}}h\bl_{\xi_{\sigma(n+1)}}h\cdots h\bl_{\xi_{\sigma(n+k)}}\imath ,\end{equation*}
is a homotopy fiber product of the diagram 
\[ \xymatrix{&\Hom^*(H_X^{n,\ast},H_X^{0<\ast<n,\ast})\ar[d] \\ KS_X[1]\ar[r]^-{\mathcal{P}_\infty} & \Hom^*(H_X^{n,\ast},H_X^{<n,\ast})} \]
and thus a model of $\Yuk_X[1]$. \end{theorem}

Fixing the above choice of a model of $\Yuk_X[1]$, the Maurer-Cartan functor becomes 
\begin{multline*}
\MC_{\Yuk_X[1]}\colon\Art\rh\Set,\quad B\mapsto\MC_{\Yuk_X[1]}(B)=\\=\left\{(s^{-1}\xi,sf)\in\left(\Yuk_X[1]\otimes\mathfrak{m}_B\right)^0\;\middle|\; \debar{\xi}=\frac{1}{2}[\xi,\xi],\,\,\, \sum_{k\geq0}\pi\left(\frac{\bi_\xi^n}{n!}\right)(h\bl_\xi)^k\imath = \pi\left(\frac{\bi_\xi^n}{n!}\right)\imath_\xi=0\right\}
\end{multline*}
thus recovering the equation of Corollary~\ref{cor.determinantal2}.

The previous model of $\Yuk_X\to KS_X$ has the advantage to be a fibration of $L_{\infty}$ algebras 
with minimal fiber. On the other hand, it has the disadvantage that the formulas for the brackets involve Green's operator, 
which is almost never explicitly known. We conclude this section by giving a second, more treatable, model of $\Yuk_X[1]$. We replace our algebraic model of the period map by  the weakly equivalent one from Theorem~\ref{th:modelperiod2}, then we apply again Theorem~\ref{thm.modelfiberproduct} to obtain an explicit model for the homotopy fiber product of
\[\xymatrix{&\Hom^*(A_X^{n,\ast},A_X^{0<\ast<n,\ast})\ar[d] \\ 
KS_X[1]\ar[r]^-{\Pi_\infty} & \Hom^*(A_X^{n,\ast},A_X^{<n,\ast})}\]

Recall that the  DG-Lie algebra structure on $(\Hom^*(A_X^{n,\ast},A_X^{<n,\ast})[-1],\delta,[-,-])$ is given by 
\[ \delta(sf)=s(-[\debar,f]-P\de f),\qquad [sf_1,sf_2]=s\left( f_1\de f_2 - (-1)^{(|f_1|+1)(|f_2|+1)}f_2\de f_1\right),\]
where $P\colon A_X^{\ast,\ast}\to A_X^{<n,\ast}$ is the projection with kernel $A_X^{n,\ast}$, and 
we identify every element $f\in\Hom^*(A_X^{n,\ast},A_X^{<n,\ast})$ with its image in $\End(A_X^{\ast,\ast})$; in particular 
$\operatorname{Im}(f)\subseteq A_X^{<n,\ast}\subseteq\operatorname{Ker}(f)$. 
For degree reasons we have $[sf_1,sf_2]=0$ whenever $f_1,f_2\in\Hom^*(A_X^{n,\ast},A_X^{0,\ast})$ and therefore the decomposition 
\[\Hom^*(A_X^{n,\ast},A_X^{<n,\ast})[-1]=\Hom^*(A_X^{n,\ast},A_X^{0<\ast<n,\ast})[-1]\oplus \Hom^*(A_X^{n,\ast},A_X^{0,\ast})[-1] \] 
satisfies the hypotheses of Voronov construction of higher derived brackets \cite{voronov,voronov2}.

\begin{lemma} The $L_\infty[1]$ algebra structure on $\Hom^*(A_X^{n,\ast},A_X^{0,\ast})[-1]$ induced via higher derived brackets is abelian with differential $q_1(sf)=-s[\debar,f]$.\end{lemma}

\begin{proof} We claim that $[\delta(sf_1),sf_2]=0$ in the DG-Lie algebra $\Hom^*(A_X^{n,\ast},A_X^{<n,\ast})[-1]$ whenever $f_1,f_2\in\Hom^*(A_X^{n,\ast},A_X^{0,\ast})$, which obviously proves the lemma. This is clear by degree reasons if $n\geq3$, while for $n=2$ we have
\[ [\delta(sf_1),sf_2]=[s(-[\debar,f_1]-\de f_1),sf_2] = (-1)^{|f_1|(|f_2|+1)}s(f_2\de^2 f_1) =0.\]
\end{proof}

\begin{theorem} The $L_\infty[1]$ algebra $(KS_X[1]\times\Hom^*(A_X^{n,\ast},A_X^{0,\ast})[-1],q_1,\ldots,q_n,\ldots)$, where the only nontrivial brackets are $q_1(s^{-1}\xi,sf)=(s^{-1}\debar\xi,-s[\debar,f])$
\[ q_2(s^{-1}\xi_1\odot s^{-1}\xi_2)=\left\{\begin{array}{ll}
((-1)^{|\xi_1|}s^{-1}[\xi_1,\xi_2],s(\bi_{\xi_1}\bi_{\xi_2})) & \mbox{if $n=2$,}\\
(-1)^{|\xi_1|}s^{-1}[\xi_1,\xi_2] & \mbox{if $n >2$,}
\end{array}\right.\]
\[ q_2(sf\otimes s^{-1}\xi)= (-1)^{|f|}s(f\bl_\xi),\]
\[q_n(s^{-1}\xi_1\odot\cdots\odot s^{-1}\xi_n)=s(\bi_{\xi_1}\cdots\bi_{\xi_n}),\]
is a homotopy fiber product of the diagram 
\[ \xymatrix{&\Hom^*(A_X^{n,\ast},A_X^{0<\ast<n,\ast})\ar[d] \\ KS_X[1]\ar[r]^-{\Pi_\infty} & \Hom^*(A_X^{n,\ast},A_X^{<n,\ast})} \]
and a model for $\Yuk_X[1]$.
\end{theorem}

\begin{proof} We have to show that the $L_\infty[1]$ algebra structure in the statement  is the one given by the formulas of Theorem~\ref{thm.modelfiberproduct}. This is a consequence of the fact that  nested brackets of the form $[\cdots[s(\bi_{\xi_1}\cdots\bi_{\xi_j}),sf_1].\cdots,sf_k]$, where $f_1,\ldots,f_k\in\Hom^*(A_X^{n,\ast},A_X^{0,\ast})$ and the brackets are computed in the DG-Lie algebra  $\Hom^*(A_X^{n,\ast},A_X^{<n,\ast})[-1]$, are trivial for obvious degree reasons except in the case $j=k=1$, where 
\[ [s(\bi_\xi),sf]=(-1)^{|\xi||f|+|\xi|+1}s(f\de\bi_\xi)=(-1)^{|\xi||f|+|\xi|+1}s(f\bl_\xi)\in\Hom^*(A_X^{n,\ast},A_X^{0,\ast})[-1],\]
	as $\bl_{\xi|A_X^{n,\ast}}=\de\bi_{\xi|A_X^{n,\ast}}$. The restriction of the $L_\infty[1]$ structure to the fiber $\Hom^*(A_X^{n,\ast},A_X^{0,\ast})[-1]$ is homotopy abelian by the previous lemma. Finally, it is clear by degree reason that the only bracket $q_k(s^{-1}\xi_1\odot\cdots\odot s^{-1}\xi_k)$ with a nontrivial component along $\Hom^*(A_X^{n,\ast},A_X^{0,\ast})[-1]$ is $q_n$.
\end{proof}

\bigskip
\section{Formality of Yukawa algebras  for K3 surfaces}

When $X$ is a compact K\"{a}hler manifold with trivial canonical bundle, we get a consistent simplification of our formulas. If $n=\dim X$, let 
$\Omega\in H^0(X,K_X)$ be a holomorphic volume form. The  Bogomolov-Tian-Todorov theorem says that the Kodaira-Spencer algebra $KS_X$ is homotopy abelian, where, according to \cite[p. 357]{GoMil2} (cf. also 
\cite[Section 7.3]{ManRendiconti}), 
an explicit homotopy equivalence 
is given by the pair of quasi-isomorphisms  of DG-Lie algebras
\[ KS_X\xleftarrow{\;\alpha\;}L=\{\xi\in KS_X\mid \bl_{\xi}(\Omega)=0\}\xrightarrow{\;\beta\;}M=
\frac{L}{I}\,,\]
where $I$ is the differential Lie ideal
\[ I=\{\xi\in KS_X\mid \bi_{\xi}(\Omega)\in \de(A^{n-2,*})\}\;.\]
More precisely,  the Koszul-Tian-Todorov lemma (\cite[Prop. 2.3]{koszul85}, \cite[Lemma 3.1]{Tian}, 
\cite[Lemma 1.2.4]{Todorov}) tells that both $L,M$ are DG-Lie algebras and the bracket on $M$ is trivial, while the $\de\debar$-lemma implies that both arrows are quasi-isomorphisms and the differential on $M$ is trivial.

Taking the pull-back of the fibration of $L_{\infty}[1]$-algebras
\[ KS_X[1]\times\Hom(H_X^{n,\ast},H_X^{0,\ast})[-1]\to KS_X[1]\]
via the strict quasi-isomorphism $\alpha$ we get a strict quasi-isomorphism 
\[ L\times\Hom(H_X^{n,\ast},H_X^{0,\ast})[-1]\to KS_X[1]\times\Hom(H_X^{n,\ast},H_X^{0,\ast})[-1]\]
in which every Taylor coefficient of degree $>n$ in $L\times\Hom(H_X^{n,\ast},H_X^{0,\ast})[-1]$ vanishes.
More precisely the nontrivial brackets are $q_1(s^{-1}\xi,sf)=(s^{-1}\debar\xi,0)$,
\[ q_2(s^{-1}\xi_1\odot s^{-1}\xi_2)=\left\{\begin{array}{ll}
((-1)^{|\xi_1|}s^{-1}[\xi_1,\xi_2],s\pi\bi_{\xi_1}\bi_{\xi_2}\imath) & \mbox{if $n=2$,}\\
((-1)^{|\xi_1|}s^{-1}[\xi_1,\xi_2],0) & \mbox{if $n >2$,}
\end{array}\right. \]
$q_k=0$ for $2<k<n$ and finally
\[q_{n}(s^{-1}\xi_1\odot\cdots\odot s^{-1}\xi_{n})=
s\pi\bi_{\xi_{1}}\cdots\bi_{\xi_{n}}\,.\]

\begin{lemma} In the above notation, if $X$ is a K3 surface then, for every $\xi\in L$ and every $\eta\in I$ we have 
\[ \pi\bi_{\xi}\bi_{\eta}=\pi\bi_{\eta}\bi_{\xi}=0\;.\]
\end{lemma}

\begin{proof} Notice first that $\bi_{\xi}\bi_{\eta}=\pm\bi_{\eta}\bi_{\xi}$ and then it is sufficient to 
prove the first equality $\pi\bi_{\xi}\bi_{\eta}=0$. 
Since $H^{2,1}_X=H^{0,1}_X=0$ the lemma is trivially verified when $|\xi|+|\eta|\not=0,2$.\par

If $|\xi|+|\eta|=2$, by Serre duality it is sufficient to show that 
\[ \int_X \pi\bi_{\xi}\bi_{\eta}(\Omega)\wedge \Omega=0\;.\]
Since $\Omega$ is $\debar$-closed and $\pi\bi_{\xi}\bi_{\eta}(\Omega)-\bi_{\xi}\bi_{\eta}(\Omega)$ is $\debar$-exact we have 
\[ \int_X \pi\bi_{\xi}\bi_{\eta}(\Omega)\wedge \Omega=
\int_X \bi_{\xi}\bi_{\eta}(\Omega)\wedge \Omega=-\int_X \bi_{\eta}(\Omega)\wedge \bi_{\xi}(\Omega)\,,\]
where the second equality follows from the fact that $\bi_{\xi}$ is a derivation of degree $0$ of the de 
Rham algebra. The last integral is trivial since by assumption $\bi_{\xi}(\Omega)$ is $\de$-closed and 
$\bi_{\eta}(\Omega)$ is $\de$-exact.\par

If $|\xi|+|\eta|=0$, again by Serre duality it is sufficient to show that 
\[ \int_X \pi\bi_{\xi}\bi_{\eta}(\Omega)\wedge \Omega\wedge\overline{\Omega}=0,\quad 
\int_X \pi\bi_{\xi}\bi_{\eta}(\Omega\wedge\overline{\Omega})\wedge \Omega=0
\;.\]
Since $\bi_{\rho}(\overline{\Omega})=0$ for every $\rho$, 
the same argument used in the previous case implies that 
\[ \int_X \pi\bi_{\xi}\bi_{\eta}(\Omega\wedge\overline{\Omega})\wedge \Omega=
\int_X \bi_{\xi}\bi_{\eta}(\Omega\wedge\overline{\Omega})\wedge \Omega=
-\int_X \bi_{\eta}(\Omega\wedge\overline{\Omega})\wedge \bi_{\xi}(\Omega)\]
and the last integral vanishes since $\bi_{\eta}(\Omega\wedge\overline{\Omega})=\bi_{\eta}(\Omega)\wedge\overline{\Omega}$ is $\de$-exact and 
$\bi_{\xi}(\Omega)\wedge\overline{\Omega}$ is $\de$-closed. 
Since $\Omega\wedge\overline{\Omega}$ is a scalar multiple of the K\"{a}hler volume form, we have 
$\debar^*(\Omega\wedge\overline{\Omega})=0$ and, since $\pi\bi_{\xi}\bi_{\eta}(\Omega)-
\bi_{\xi}\bi_{\eta}(\Omega)$ is $\debar^*$-exact, we have 
\[ \int_X \pi\bi_{\xi}\bi_{\eta}(\Omega)\wedge \Omega\wedge\overline{\Omega}=
\int_X \bi_{\xi}\bi_{\eta}(\Omega)\wedge \Omega\wedge\overline{\Omega}\]
and the same argument as above shows that the last integral vanishes.

\end{proof}

\begin{theorem} The Yukawa $L_{\infty}$ algebra $\Yuk_X$ of a K3 surface $X$ is formal. 
More precisely, $\Yuk_X$ is quasi-isomorphic to the graded Lie algebra
\[ H^*(X,\Theta_X)\times \Hom^*(H^*(\Omega^2_X),H^*(\Oh_X)),\qquad [(\xi,u),(\eta,v)]=(-1)^{|\xi|}\bi_{\xi}\bi_{\eta}\,.\]
\end{theorem}

\begin{proof} 
By the previous lemma the subspace $I[1]\subset L[1]\times\Hom(H_X^{n,\ast},H_X^{0,\ast})[-1]$ is an $L_{\infty}[1]$ ideal and then the projection 
\[ L[1]\times\Hom(H_X^{n,\ast},H_X^{0,\ast})[-1]\to \frac{L[1]}{I[1]}\times\Hom(H_X^{n,\ast},H_X^{0,\ast})[-1]\]
is a strict quasi-isomorphism. 
Since $\Hom(H_X^{n,\ast},H_X^{0,\ast})[-1]=\Hom(H_X^{n,\ast}[2],H_X^{0,\ast})[1]$ 
it is sufficient to rewrite the induced $L_{\infty}[1]$ structure on the quotient 
via the natural isomorphisms 
\[\frac{L[1]}{I[1]}\simeq H^*(X,\Theta_X)[1],\quad H^{n,*}_X[2]\simeq H^*(X,\Omega^2_X),
\quad H^{0,*}_X\simeq H^*(X,\Oh_X)\]
and then apply the d\'ecalage functor.
\end{proof}

\section*{Acknowledgments} Both authors 
acknowledge partial support by Italian MIUR under PRIN project 2012KNL88Y ``Spazi di moduli e teoria di Lie''.

\end{document}